\documentclass[11pt,letterpaper]{amsart}
\pdfoutput=1

\usepackage[T1]{fontenc}
\usepackage[utf8]{inputenc}
\usepackage[english]{babel}
\usepackage{lmodern}
\usepackage[section]{placeins}
\usepackage{float}
\usepackage{geometry}
\usepackage{hyperref}
\usepackage{booktabs}
\usepackage{enumitem}
\usepackage{url}
\usepackage{subcaption}
\usepackage{xcolor}
\usepackage{colortbl}
\definecolor{darkgreen}{rgb}{0.33, 0.55, 0.13}
\definecolor{darkpink}{rgb}{0.91, 0.33, 0.5}
\definecolor{electriccyan}{rgb}{0.0, 1.0, 1.0}
\definecolor{electricultramarine}{rgb}{0.25, 0.0, 1.0}
\definecolor{greenyellow}{rgb}{0.30, 0.82, 0.0}
\definecolor{lightbrown}{rgb}{0.71, 0.4, 0.11}
\definecolor{prune}{rgb}{0.45, 0.11, 0.11}
\definecolor{mygold}{RGB}{212, 175, 55}
\definecolor{levired}{RGB}{220, 50, 50}   
\definecolor{leviblue}{RGB}{50, 100, 220} 
\definecolor{levigold}{RGB}{218, 165, 32} 
\definecolor{crtgold}{rgb}{0.8, 0.6, 0.0} 

\usepackage{amsmath}
\usepackage{amssymb}
\usepackage{amsthm}
\usepackage{caption}
\usepackage{tikz}
\usetikzlibrary{3d,arrows.meta,calc,decorations.markings,positioning,intersections}

\newcommand{\Z}{\mathbb{Z}}

\hypersetup{
    colorlinks=true,
    linkcolor=blue,
    urlcolor=blue,
    pdftitle={Can one decompose a simple tone?}
}
\newtheorem{theorem}{Theorem}
\newtheorem{definition}{Definition}

\newtheorem{conclusion}{Conclusion}

\newtheorem{remark}{Remark}

\begin{document}

\title[WHAT IF WE DECOMPOSE A SIMPLE TONE?]{WHAT IF WE DECOMPOSE A SIMPLE TONE? \\[0.5ex] THE CHINESE REMAINDER THEOREM AND STRUCTURED LEVI GRAPHS IN MUSIC}

\author{Pawe{\l} Nurowski}
\address{Centrum Fizyki Teoretycznej,
Polska Akademia Nauk, Al.
Lotnik\'ow 32/46, 02-668 Warszawa, Poland, and
Guangdong Technion -- Israel Institute of Technology, No. 241, Daxue Road,
Jinping District, Shantou, Guangdong Province, China}
\email{nurowski@cft.edu.pl}

\maketitle

\begin{abstract}
While motivated by structural problems in mathematical music theory, this article introduces a novel combinatorial framework that advances the classification of cyclic cubic bipartite graphs. We extend the classical study of Levi graphs by endowing their vertices with an internal algebraic anatomy—specifically, treating them not as empty geometric nodes, but as defined subsets of a cyclic base space $\mathbb{Z}_n$. This internal structure allows us to formalize and classify a highly restricted class of graph isomorphisms: those strictly induced by global affine bijections $f(x) \equiv ax+b \pmod n$ operating directly on the underlying base set. 
By applying this framework to generalized tone networks (Tonnetze) unrolled via the Chinese Remainder Theorem in composite dimensions—specifically the classic 12-TET ($3\times4$) and the decaphonic 10-TET ($2\times5$)—we reveal absolute geometric anchors for these spaces, namely the (9,4) and (6,5) systems respectively. We completely classify the topological orbits of these structured graphs, proving a fundamental architectural dichotomy: while the isomorphic landscape of 12-TET splits into an orientation-preserving family and an orientation-reversing chiral mirror (providing a rigorous foundation for musical Negative Harmony), the 10-TET space is unconditionally orientation-preserving. Finally, we demonstrate that these abstract combinatorial properties manifest as rigidly coherent, parallel auditory universes through explicit structural voice-leading maps and acoustic physical modeling synthesis.
\end{abstract}

\section{Introduction}
\label{sec:intro}

The interplay between combinatorics and structural algebra often reveals hidden topological symmetries within spaces generated by cyclic group actions. While the ultimate inquiry of this paper is motivated by a concrete problem in mathematical music theory—namely, the mapping, navigation, and acoustic synthesis of functional chord progressions across generalized tuning systems—the theoretical apparatus developed herein addresses a fundamental question in graph theory. Specifically, we investigate the classification of cyclic cubic bipartite Levi graphs whose vertices are equipped with an internal mathematical anatomy, being treated as structured subsets of a discrete base space $\mathbb{Z}_n$. By focusing on graph isomorphisms that are directly induced by the algebraic bijections of this underlying base set, we establish a robust topological grammar that separates invariant structural properties from the specific vocabulary of the generating elements.

Our inquiry begins not with an abstract mathematical axiom, but with a direct observation of the familiar Western major triad. Traditionally, this chord is constructed by partitioning a Perfect Fifth ($q = 7$ semitones) into a Major Third ($t = 4$) and a Minor Third ($s = 3$). Rather than accepting these specific intervals as immovable acoustic dogma, we can treat them as algebraic variables. What if we maintain the fundamental structural relation $t + s \equiv q \pmod{n}$, but allow the intervals, and indeed the entire size of the musical universe $n$, to vary?

In our previous work \cite{nur1}, we formalised this concept by introducing a generalised framework for harmony in any $n$-tone equal temperament ($n$-TET). By identifying the $n$ pitch classes with the elements of the cyclic group $\Z_n$ under modulo $n$ addition, we defined a generic harmonic system denoted by $(t,s)$, where $0 < s, t \le n-1$. For any pitch class $r \in \Z_n$, a `Major chord' is defined as the mathematical triple $M_r = (r, r+t, r+q)$, whilst a `Minor chord' is defined as the triple $m_r = (r, r+s, r+q)$.

To understand how music flows within such a system, we constructed a generalised tone network---a \textit{Tonnetz} in the historical spirit of Euler \cite{euler_tentamen} and Riemann \cite{riemann_ideen}, and its modern formalization by Cohn \cite{cohn_audacious}. A Major chord $M_r$ connects to exactly three adjacent Minor chords via the transformations $M_r \to (m_r, m_{r+t}, m_{r-s})$. This naturally forces the reciprocal mapping for Minor chords: $m_r \to (M_r, M_{r-t}, M_{r+s})$. Musically speaking, a `good' chord progression is simply a path that gracefully traverses the allowed connections of this Tonnetz.

Mathematically, this network manifests as a bipartite Levi graph, a perspective recently formalised and expanded upon by Boland and Hughston \cite{boland_hughston_2025, boland_hughston_2026}, who demonstrated that various tone networks can be uniquely represented as finite geometric configurations. We can visualise it beautifully on a circle containing 2n equidistant points, where the $n$ Major chords act as `Points' and the $n$ Minor chords act as `Lines'. The Tonnetz connections are simply the line segments drawn between them, with each Point connecting to three Lines, and each Line connecting to three Points.

Applying this topological lens to the standard 12-TET reveals a striking hidden landscape. We demonstrated in \cite{nur1} that the Levi graphs for all possible $(t,s)$ systems in $n=12$ collapse into exactly eight distinct topologies. The classical Western system, $(4,3)$, resides in a topology shared by eleven other systems. If we discard the degenerate systems where the generator $q$ shares a common divisor with $n$ (as this geometrically shatters the pitch space into disconnected sub-universes), we are left with a fascinating quartet. The classical $(4,3)$ system is topologically isomorphic to only three others: $(9,4)$, $(8,3)$, and $(9,8)$.

This raises a profound question: if these four systems share the exact same mathematical geometry, why does standard Western harmony exclusively utilise $(4,3)$? The short answer is acoustic resonance---it best approximates the natural harmonic series. However, this prompts an even deeper mystery: what is the true purpose of the other systems in this topological orbit? 

As it turns out, the exotic $(9,4)$ system is not merely a mathematical curiosity; it is the fundamental anchor---the absolute algebraic baseline---from which all other $(t,s)$ systems in 12-TET are derived. Similarly, as explored in \cite{nur2}, within the decaphonic universe of 10-TET, this fundamental anchor is the $(6,5)$ system. 

The primary aim of this article is to reveal the origin of these fundamental systems, such as $(9,4)$ for $n=12$ and $(6,5)$ for $n=10$, and to demonstrate how they serve as the master blueprints for defining elegant harmony in exotic temperaments. The key to unlocking this lies in the Chinese Remainder Theorem. As hinted at the conclusion of \cite{nur1}, when $n$ is a product of two coprime integers, a given pitch class is not an indivisible atom. Rather, it possesses internal structure. A pitch class is inherently a two-dimensional, composite object. It is this pivotal realisation---that one can indeed decompose a simple tone---that will unravel the topological mysteries of our musical universes.

\section{Decomposing the Tone: The Chinese Remainder Theorem}
\label{sec:crt}

To understand the topological orbits mentioned in the introduction, we must stop viewing musical tones as indivisible dots on a circle. A single pitch class within an $n$-tone temperament is not an atomic entity; provided $n$ is composite, the tone possesses an internal, multi-dimensional structure. In this section, we will reveal this structure using a classical mathematical tool. We shall begin by demonstrating this concept within the universally familiar 12-TET system, before formalising the mathematics and applying it to uncover the topology of the decaphonic 10-TET universe.

\subsection{Unrolling the 12-TET Clock Face}

Let us initiate our geometric exploration with the standard 12-tone system. The number 12 possesses a specific composite structure: it is the product of 3 and 4. Because 3 and 4 share no common factors, we can take the 1-dimensional ``clock face'' of 12 notes and unroll it into a 2-dimensional grid of 3 columns and 4 rows. Mathematicians call this property an isomorphism between the cyclic group $\Z_{12}$ and the direct product $\Z_3 \times \Z_4$.

This is not merely an abstract equivalence; it provides a specific recipe for mapping every musical note $x \in \{0, \dots, 11\}$ to a unique coordinate pair $(u, v)$ on a grid. The rule is simple: the first coordinate is the remainder when you divide the note by 3, and the second is the remainder when you divide the note by 4:
\begin{equation}
x \longmapsto (x \bmod 3, \;\; x \bmod 4).
\end{equation}
Let us calculate the position of a few key musical intervals to see how this flat geometry constructs itself:
\begin{itemize}
    \item The root note 0 (C) maps to $(0 \bmod 3, 0 \bmod 4) = (0,0)$. This is the origin of our grid.
    \item The semitone {\color{crtgold}1} ({\color{crtgold}C$^\sharp$}) maps to ${\color{crtgold}(1,1)}$. Because it steps equally in both directions, it forms the diagonal path of the grid.
    \item The Major Third {\color{blue}4} ({\color{blue}E}) maps to $(4 \bmod 3, 4 \bmod 4) = {\color{blue}(1, 0)}$. This is the fundamental step (basis vector) moving purely along the horizontal $\Z_3$ axis.
    \item The Major Sixth {\color{red}9} ({\color{red}A}) maps to $(9 \bmod 3, 9 \bmod 4) = {\color{red}(0, 1)}$. This is the fundamental step (basis vector) moving purely along the vertical $\Z_4$ axis.
\end{itemize}
Mapping the entire chromatic scale onto this flat product space reveals the hidden geometric grid visualised in Figure \ref{fig:crt_grid_12tet}.

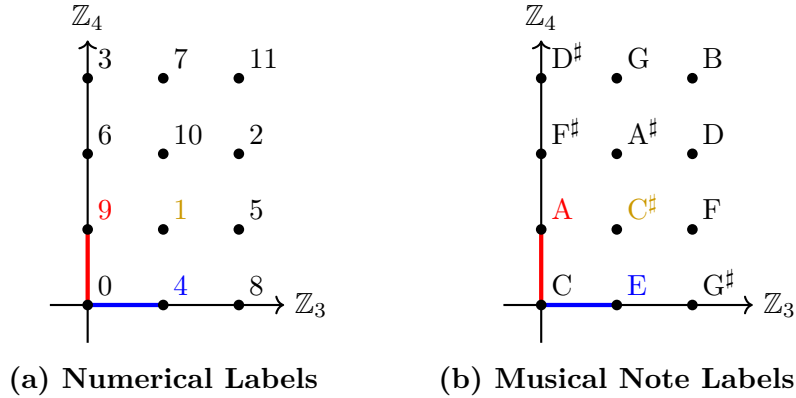
\begin{figure}[H]
    \centering
    \begin{tikzpicture}[scale=1.0]
        \begin{scope}
            \draw[->, thick] (-0.5,0) -- (2.6,0) node[right] {$\Z_3$};
            \draw[->, thick] (0,-0.5) -- (0,3.5) node[above] {$\Z_4$};
            
            \draw[-, ultra thick, blue] (0,0) -- (1,0);
            \draw[-, ultra thick, red] (0,0) -- (0,1);

            \fill (0,0) circle (2pt) node[above right] {0};
            \fill (1,0) circle (2pt) node[above right] {{\color{blue}4}};
            \fill (2,0) circle (2pt) node[above right] {8};
            
            \fill (0,1) circle (2pt) node[above right] {{\color{red}9}};
            \fill (1,1) circle (2pt) node[above right] {{\color{crtgold}1}};
            \fill (2,1) circle (2pt) node[above right] {5};
            
            \fill (0,2) circle (2pt) node[above right] {6};
            \fill (1,2) circle (2pt) node[above right] {10};
            \fill (2,2) circle (2pt) node[above right] {2};
            
            \fill (0,3) circle (2pt) node[above right] {3};
            \fill (1,3) circle (2pt) node[above right] {7};
            \fill (2,3) circle (2pt) node[above right] {11};
            
            \node at (1, -1) {\textbf{(a) Numerical Labels}};
        \end{scope}

        \begin{scope}[xshift=6cm]
            \draw[->, thick] (-0.5,0) -- (2.8,0) node[right] {$\Z_3$};
            \draw[->, thick] (0,-0.5) -- (0,3.5) node[above] {$\Z_4$};
            
            \draw[-, ultra thick, blue] (0,0) -- (1,0);
            \draw[-, ultra thick, red] (0,0) -- (0,1);
            
            \fill (0,0) circle (2pt) node[above right] {C};
            \fill (1,0) circle (2pt) node[above right] {{\color{blue}E}};
            \fill (2,0) circle (2pt) node[above right] {G$^\sharp$};
            
            \fill (0,1) circle (2pt) node[above right] {{\color{red}A}};
            \fill (1,1) circle (2pt) node[above right] {{\color{crtgold}C$^\sharp$}};
            \fill (2,1) circle (2pt) node[above right] {F};
            
            \fill (0,2) circle (2pt) node[above right] {F$^\sharp$};
            \fill (1,2) circle (2pt) node[above right] {A$^\sharp$};
            \fill (2,2) circle (2pt) node[above right] {D};
            
            \fill (0,3) circle (2pt) node[above right] {D$^\sharp$};
            \fill (1,3) circle (2pt) node[above right] {G};
            \fill (2,3) circle (2pt) node[above right] {B};
            
            \node at (1, -1) {\textbf{(b) Musical Note Labels}};
        \end{scope}
    \end{tikzpicture}
    \caption{The decomposition of the 12-tone universe onto the $\Z_3 \times \Z_4$ grid. The blue segment represents the interval {\color{blue}4} (Major Third), and the red segment represents the interval {\color{red}9} (Major Sixth). The traditional `unison' step of a semitone is located at {\color{crtgold}(1,1)}.}
    \label{fig:crt_grid_12tet}
\end{figure}

The intervals 4 and 9 play a distinguished, foundational role in this geometry. Because they map exactly to the coordinates $(1,0)$ and $(0,1)$, they act as the minimal generators---the pure horizontal and vertical basis vectors---of the entire 12-tone grid. This explains mathematically why the $(9,4)$ system acts as the absolute anchor for 12-TET harmony: it is simply the coordinate framework of the pitch space itself.

\subsection{The Formal Mathematical Engine: Chinese Remainder Theorem}

The geometric unrolling we have just witnessed is driven by one of the most elegant results in number theory. We state it here in its general algebraic form:

\begin{theorem}[Chinese Remainder Theorem]
Let $n_1, n_2, \dots, n_k$ be positive integers that are pairwise coprime (i.e., $\gcd(n_i, n_j) = 1$ for all $i \neq j$). Let $n = n_1 n_2 \cdots n_k$. Then the map
\begin{equation}
    \Phi: \Z_n \longrightarrow \Z_{n_1} \times \Z_{n_2} \times \dots \times \Z_{n_k}
\end{equation}
defined by $\Phi(x) = (x \bmod n_1, x \bmod n_2, \dots, x \bmod n_k)$ is a ring isomorphism.
\end{theorem}

In the context of music theory, where we are primarily concerned with the cyclic group of transpositions under addition, this theorem guarantees two profound properties. Firstly, the map is a \textit{bijection}---meaning every single pitch class maps to one and only one coordinate on the grid, with no empty spaces and no overlaps. Secondly, it is a \textit{homomorphism}---meaning that adding intervals on the 1D clock face is exactly equivalent to adding vectors on the 2D grid. The musical structure is perfectly preserved.

For any bipartite system where $n = p \cdot q$ with $\gcd(p,q) = 1$, the fundamental harmonic generators are always the specific intervals $x_p$ and $x_q$ that satisfy:
\begin{equation}
    \Phi(x_p) = (1, 0) \quad \text{and} \quad \Phi(x_q) = (0, 1).
\end{equation}
These `primitive tones' are the architectural pillars of the tuning system.

\subsection{The Decaphonic Grid: Application to 10-TET}

Armed with this theorem, we can now venture into the exotic territory of 10-TET. The decaphonic scale contains $n=10$ notes. Its prime factorisation is $10 = 2 \times 5$. Since 2 and 5 are coprime ($\gcd(2,5)=1$), the Chinese Remainder Theorem dictates that $\Z_{10} \cong \Z_2 \times \Z_5$. 

We decompose the 10-TET tones by mapping them to coordinates $(x \bmod 2, x \bmod 5)$. The space forms a narrow grid of 2 columns and 5 rows. To find the foundational anchors of this universe, we must locate the basis vectors $(1,0)$ and $(0,1)$:
\begin{itemize}
    \item We seek an interval $x$ such that $x \equiv 1 \pmod 2$ and $x \equiv 0 \pmod 5$. The solution is $x = 5$. Thus, {\color{blue}5} maps to {\color{blue}(1,0)}.
    \item We seek an interval $x$ such that $x \equiv 0 \pmod 2$ and $x \equiv 1 \pmod 5$. The solution is $x = 6$. Thus, {\color{red}6} maps to {\color{red}(0,1)}.
\end{itemize}
Just as 4 and 9 generate the 12-TET grid, the intervals 5 and 6 act as the absolute basis vectors for 10-TET. The 1-step (120 cents, analogous to a semitone) maps to the diagonal $(1,1)$.

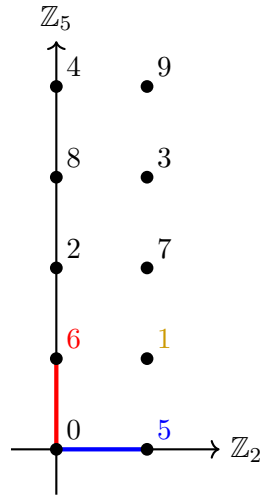
\begin{figure}[H]
    \centering
    \begin{tikzpicture}[scale=1.2]
        \draw[->, thick] (-0.5,0) -- (1.8,0) node[right] {$\Z_2$};
        \draw[->, thick] (0,-0.5) -- (0,4.5) node[above] {$\Z_5$};

        \draw[-, ultra thick, blue] (0,0) -- (1,0);
        \draw[-, ultra thick, red] (0,0) -- (0,1);

        \fill (0,0) circle (2pt) node[above right] {0};
        \fill (1,0) circle (2pt) node[above right] {{\color{blue}5}};

        \fill (0,1) circle (2pt) node[above right] {{\color{red}6}};
        \fill (1,1) circle (2pt) node[above right] {{\color{crtgold}1}};

        \fill (0,2) circle (2pt) node[above right] {2};
        \fill (1,2) circle (2pt) node[above right] {7};

        \fill (0,3) circle (2pt) node[above right] {8};
        \fill (1,3) circle (2pt) node[above right] {3};

        \fill (0,4) circle (2pt) node[above right] {4};
        \fill (1,4) circle (2pt) node[above right] {9};
    \end{tikzpicture}
    \caption{The decaphonic grid: mapping the 10-TET pitch classes onto $\Z_2 \times \Z_5$. The primitive tones forming the independent horizontal and vertical axes are {\color{blue}5} and {\color{red}6}. This mathematical reality establishes the $(6,5)$ system as the fundamental baseline for 10-TET harmony.}
    \label{fig:crt_grid_10tet}
\end{figure}

The topological implication is profound. Without any trial and error, the mathematics definitively selects $(6,5)$ as the canonical anchor for a decaphonic tuning system. Every harmonic progression and isomorphic orbit in 10-TET must ultimately be referenced back to this underlying $\Z_2 \times \Z_5$ structural lattice.

\section{Structured Networks and Note-Induced Isomorphisms}
\label{sec:isomorphisms}

As established in Section \ref{sec:crt}, the Chinese Remainder Theorem reveals the foundational geometric anchors of our musical spaces, namely the $(9,4)$ system for 12-TET and the $(6,5)$ system for 10-TET. Yet, as noted in the introduction, the classical Western system $(4,3)$ resides in the exact same topological orbit as the $(9,4)$ anchor. To understand how the acoustic standard is mathematically bound to this baseline, we must formally examine the internal anatomy of our Tonnetz networks.

\subsection{Abstract Levi Graphs and Isomorphisms}

We begin by formally defining the abstract geometric structure of a harmonic system.

\begin{definition}[Levi Graph of a Harmonic System]
For a given harmonic system $(t,s)$ in $\Z_n$, its \textbf{Levi graph} is a bipartite graph $\mathcal{G}(t,s) = (\mathcal{P} \cup \mathcal{L}, E)$, where the disjoint vertex sets $\mathcal{P}$ and $\mathcal{L}$ represent the $n$ Major chords and $n$ Minor chords, respectively. An edge $e \in E$ exists between a vertex in $\mathcal{P}$ and a vertex in $\mathcal{L}$ if and only if the corresponding chords share exactly two pitch classes, generating a 3-regular bipartite network.
\end{definition}

In the context of incidence geometry, such structures naturally correspond to the incidence graphs of symmetric configurations, the enumeration and classification of which---particularly for small orders like $10_3$ and $12_3$---have been extensively studied \cite{betten_configurations, alazemi_12_3}. Furthermore, the specific representation of the Tonnetz as a finite Levi graph has recently proven to be a highly effective tool in identifying the underlying combinatorial geometries of abstract musical systems, a perspective formalized and significantly expanded upon in the recent works of Boland and Hughston \cite{boland_hughston_2025, boland_hughston_2026}.

To mathematically compare the topologies of two different harmonic systems, we rely on the standard graph-theoretic notion of isomorphism, keeping track of the bipartite partitions.



\begin{definition}[Isomorphism of Levi Graphs]
Two Levi graphs $\mathcal{G}(t,s) = (\mathcal{P} \cup \mathcal{L}, E)$ and $\mathcal{G}(t',s') = (\mathcal{P}' \cup \mathcal{L}', E')$ are \textbf{abstractly isomorphic} if there exists a bijection $\phi: \mathcal{P} \cup \mathcal{L} \to \mathcal{P}' \cup \mathcal{L}'$ that preserves adjacency: $\{u,v\} \in E \iff \{\phi(u),\phi(v)\} \in E'$. 
\begin{itemize}
    \item The isomorphism is \textbf{orientation-preserving} if it preserves the bipartite partitions, mapping $\mathcal{P} \to \mathcal{P}'$ and $\mathcal{L} \to \mathcal{L}'$.
    \item The isomorphism is \textbf{orientation-reversing} if it swaps the partitions, mapping $\mathcal{P} \to \mathcal{L}'$ and $\mathcal{L} \to \mathcal{P}'$.
\end{itemize}
\end{definition}

\subsection{Formalising Note-Induced Isomorphisms}

In standard graph theory, an abstract isomorphism between two networks simply establishes a bijection between their vertices that preserves adjacency. However, in our musical context, the nodes of a Tonnetz are not empty geometric points. Each node represents a chord, which is a specific 3-element subset of the pitch class space $\Z_n$. Therefore, an abstract graph isomorphism between two harmonic systems does not automatically imply a meaningful musical relationship. 

The crucial mathematical question is whether the geometric equivalence of the networks can be derived directly from a single, consistent algebraic transformation of the underlying pitch classes. To rigorously calculate these transformations, we must formally link the abstract graph to the internal pitch structure of its nodes.

\begin{definition}[Structured Levi Graph]
Let $n \in \mathbb{N}$. A \textbf{structured Levi graph} over $\Z_n$ is a bipartite graph $\mathcal{G} = (\mathcal{P} \cup \mathcal{L}, E)$ equipped with an injective map $\mu: \mathcal{P} \cup \mathcal{L} \to \binom{\Z_n}{3}$, where $\binom{\Z_n}{3}$ denotes the set of all 3-element subsets of $\Z_n$. 

The map $\mu$ assigns to each vertex $P \in \mathcal{P}$ (Major chord) and each vertex $L \in \mathcal{L}$ (Minor chord) a unique subset of exactly three distinct pitch classes from $\Z_n$. An edge $e \in E$ exists between $P \in \mathcal{P}$ and $L \in \mathcal{L}$ if and only if their corresponding subsets share exactly two elements:
\begin{equation}
    |\mu(P) \cap \mu(L)| = 2.
\end{equation}
\end{definition}

When a graph isomorphism is generated by a global affine mapping of the notes themselves, we call it a \textbf{note-induced isomorphism}. 

\begin{definition}[Note-Induced Isomorphism]
Let $\mathcal{G}(t,s)$ and $\mathcal{G}(t',s')$ be isomorphic structured Levi graphs. We say this isomorphism is \textbf{induced by notes} if there exists a global affine bijection of the pitch classes $f: \Z_n \to \Z_n$ of the form:
\begin{equation}
    f(x) \equiv ax + b \pmod{n}
\end{equation}
with $\gcd(a,n)=1$ and $b \in \Z_n$, such that applying $f$ element-wise to the pitch classes of the chords in $\mathcal{G}(t,s)$ bijectively yields the corresponding chords in $\mathcal{G}(t',s')$. 

As established abstractly, $f$ is \textbf{orientation-preserving} if it strictly preserves chord polarities, mapping Major chords to Major chords ($\mathcal{P} \to \mathcal{P}'$) and Minor to Minor ($\mathcal{L} \to \mathcal{L}'$). 
$f$ is \textbf{orientation-reversing} if it reverses chord polarities, mapping Major chords to Minor chords ($\mathcal{P} \to \mathcal{L}'$) and Minor to Major ($\mathcal{L} \to \mathcal{P}'$). This acts as a chiral mirror, mathematically inverting the harmony.
\end{definition}

From an algebraic standpoint, the family of bijections of the form $f(x) \equiv ax + b \pmod{n}$ constitutes the well-known general affine group $\text{Aff}(\Z_n)$ over the ring of integers modulo $n$. In the paradigm of Mathematical Music Theory (MMT), this group and its structural actions represent a central object of investigation for formalizing generalized musical symmetries, transformation networks, and interval systems, most notably .exemplified in the seminal transformational and geometric frameworks of Lewin \cite{lewin}, Mazzola \cite{mazzola}, and Tymoczko \cite{tymoczko_geometry}.

\subsection{The 12-TET Landscape: The Orientation-Preserving Family and the Orientation-Reversing Mirror}

To determine the nature of a note-induced isomorphism, we must analyse how the affine transformation affects the internal interval structure of a chord. For any two notes $x$ and $y$, their distance under the transformation $f(x) = ax + b$ becomes $f(y) - f(x) = (ay + b) - (ax + b) \equiv a(y - x) \pmod{n}$. Crucially, the affine shift $b$ completely cancels out. This proves that \textbf{the topological nature (chirality) of the generated system is determined strictly and exclusively by the multiplier $a$}. The parameter $b$ serves merely as a global transposition allowing the transformed chord to align with the specific roots of a target system.

We now apply this rigorous framework to the classical Western system $(4,3)$ to find its dual pairs and prove their interconnections. 

\begin{theorem}[Note-Induced Isomorphisms in 12-TET]
In $\Z_{12}$, the condition of non-degeneracy yields exactly four harmonic systems with isomorphic Levi graphs: $(4,3)$, $(8,3)$, $(9,4)$, and $(9,8)$. Every graph isomorphism between any pair of these four systems can be explicitly induced by an affine note bijection $f(x) \equiv ax + b \pmod{12}$. Specifically:
\begin{enumerate}
    \item The systems $(4,3)$, $(8,3)$, and $(9,4)$ form an \textbf{orientation-preserving family}. They are mutually interconnected by affine bijections that strictly preserve chord polarities (Major maps to Major). 
    \item The system $(9,8)$ acts as the \textbf{orientation-reversing mirror}. Any affine note bijection connecting the orientation-preserving family to $(9,8)$ necessarily reverses chord polarities (Major maps to Minor).
\end{enumerate}
\end{theorem}

\begin{proof}
To prove that a Levi graph isomorphism is induced by notes, we must demonstrate an affine bijection $f(x) \equiv ax+b \pmod{12}$ that maps the pitch classes of the chords in the source system directly to the pitch classes of the chords in the target system. Let us take the classical system $(4,3)$ as our base. Its Major chord is $M_k^{(4,3)} = \{k, k+4, k+7\}$ and its Minor chord is $m_k^{(4,3)} = \{k, k+3, k+7\}$.

\textbf{1. Passage from (4,3) to (8,3):}
Let $f(x) = 5x$. Applying this element-wise to the notes of the $(4,3)$ chords yields:
\begin{align*}
    f(M_k^{(4,3)}) &= \{5k, 5k+20, 5k+35\} \equiv \{5k, 5k+8, 5k+11\} \pmod{12}. \\
    f(m_k^{(4,3)}) &= \{5k, 5k+15, 5k+35\} \equiv \{5k, 5k+3, 5k+11\} \pmod{12}.
\end{align*}
In the $(8,3)$ system, the generator is $q=11$, so its defining chords are $M_j^{(8,3)} = \{j, j+8, j+11\}$ and $m_j^{(8,3)} = \{j, j+3, j+11\}$. 
Substituting $j = 5k$, we see that $f(M_k^{(4,3)}) = M_{5k}^{(8,3)}$ and $f(m_k^{(4,3)}) = m_{5k}^{(8,3)}$. The map $f(x)=5x$ perfectly induces the Levi graph isomorphism while preserving chord polarity (orientation-preserving).

\textbf{2. Passage from (4,3) to (9,4):}
Let $f(x) = 5x+1$. Applying this to the $(4,3)$ chords yields:
\begin{align*}
    f(M_k^{(4,3)}) &= \{5k+1, 5k+21, 5k+36\} \equiv \{5k+1, 5k+9, 5k\} \pmod{12}. \\
    f(m_k^{(4,3)}) &= \{5k+1, 5k+16, 5k+36\} \equiv \{5k+1, 5k+4, 5k\} \pmod{12}.
\end{align*}
In the $(9,4)$ system, $q=1$, so $M_j^{(9,4)} = \{j, j+9, j+1\}$ and $m_j^{(9,4)} = \{j, j+4, j+1\}$. 
Letting $j = 5k$, we find $f(M_k^{(4,3)}) = M_{5k}^{(9,4)}$ and $f(m_k^{(4,3)}) = m_{5k}^{(9,4)}$. The map $f(x)=5x+1$ induces an orientation-preserving isomorphism.

\textbf{3. Passage from (4,3) to (9,8):}
Let $f(x) = 11x$. 
\begin{align*}
    f(M_k^{(4,3)}) &= \{11k, 11k+44, 11k+77\} \equiv \{11k, 11k+8, 11k+5\} \pmod{12}. 
\end{align*}
In the $(9,8)$ system, $q=5, t=9, s=8$, so a Minor chord is $m_j^{(9,8)} = \{j, j+8, j+5\}$. 
Thus, $f(M_k^{(4,3)}) = m_{11k}^{(9,8)}$. A Major chord has been mapped to a Minor chord. The map $f(x)=11x$ induces an orientation-reversing isomorphism (Negative Harmony).

\textbf{4. Passage from (8,3) to (9,8):}
To map from $(8,3)$ directly to $(9,8)$, we compose the inverse of the $(4,3)\to(8,3)$ map with the $(4,3)\to(9,8)$ map. The required multiplier is $11 \times 5^{-1} \equiv 11 \times 5 \equiv 55 \equiv 7 \pmod{12}$. Let $f(x) = 7x$.
\begin{align*}
    f(M_k^{(8,3)}) &= 7 \times \{k, k+8, k+11\} = \{7k, 7k+56, 7k+77\} \equiv \{7k, 7k+8, 7k+5\} \pmod{12}.
\end{align*}
This is precisely $m_{7k}^{(9,8)}$. Thus, $f(x)=7x$ induces an orientation-reversing isomorphism between $(8,3)$ and $(9,8)$.

By explicitly constructing the affine note bijections covering all cases, the theorem is proven.
\end{proof}

\begin{conclusion}[The 12-TET Musical Equivalence]
The existence of these affine note-induced isomorphisms carries a profound practical consequence: the four systems within this topological orbit are perfectly transposable musical universes. An entire musical composition can be systematically translated from any one of these four systems into any other using the transformations $f(x) \equiv ax+b \pmod{12}$. Because these bijections perfectly preserve the adjacencies of the Levi graph, the entire harmonic logic, voice-leading syntax, and progression structure of the original piece remain flawlessly intact. The composition is merely recast into a new acoustic reality—either preserving its original Major/Minor polarity (within the orientation-preserving family) or emerging mathematically inverted (via the orientation-reversing $(9,8)$ mirror).
\end{conclusion}

\subsection{The 10-TET Landscape: Absolute Orientation-Preservation in the Decaphonic Universe}

From Section \ref{sec:crt}, we know the Chinese Remainder Theorem establishes $(6,5)$ as the canonical baseline for 10-TET. Let us investigate the note-induced isomorphisms connecting its privileged topological quartet. 

\begin{theorem}[Note-Induced Isomorphisms in 10-TET]
In $\Z_{10}$, the non-degenerate isomorphic quartet consists of $(6,5)$, $(8,5)$, $(2,5)$, and $(4,5)$. In stark contrast to 12-TET, \textbf{every single note-induced isomorphism generating these systems is strictly orientation-preserving}. 
\end{theorem}

\begin{proof}
Let us take the fundamental baseline $(6,5)$ as our reference. Its base Major chord is $M_k^{(6,5)} = \{k, k+6, k+1\} \pmod{10}$. The allowed multipliers in $\mathrm{Aut}(\Z_{10})$ are $a \in \{1, 3, 7, 9\}$. 

We apply the linear bijection $f(x) = ax$ (setting $b=0$) to the base chord $M_k^{(6,5)}$ for each multiplier:
\begin{itemize}
    \item \textbf{Case $a = 1$:} $f(x) = x \implies M_k^{(6,5)}$.
    \item \textbf{Case $a = 3$:} $f(M_k^{(6,5)}) = \{3k, 3k+18, 3k+3\} \equiv \{3k, 3k+8, 3k+3\} \pmod{10}$. This exact set of notes constitutes the Major chord $M_{3k}^{(8,5)}$. Thus, $f(x)=3x$ is an orientation-preserving map to $(8,5)$.
    \item \textbf{Case $a = 7$:} $f(M_k^{(6,5)}) = \{7k, 7k+42, 7k+7\} \equiv \{7k, 7k+2, 7k+7\} \pmod{10}$. This exact set of notes constitutes the Major chord $M_{7k}^{(2,5)}$. Thus, $f(x)=7x$ is an orientation-preserving map to $(2,5)$.
    \item \textbf{Case $a = 9$:} $f(M_k^{(6,5)}) = \{9k, 9k+54, 9k+9\} \equiv \{9k, 9k+4, 9k+9\} \pmod{10}$. This exact set of notes constitutes the Major chord $M_{9k}^{(4,5)}$. Thus, $f(x)=9x$ is an orientation-preserving map to $(4,5)$.
\end{itemize}
Because the pure linear mappings $f(x) = ax$ perfectly map Major chords to Major chords, every multiplier in $\mathrm{Aut}(\Z_{10})$ generates an orientation-preserving isomorphism. Any affine shift $b \neq 0$ merely transposes these orientation-preserving systems up or down the pitch space.
\end{proof}

\begin{conclusion}[The Decaphonic Musical Equivalence]
The mathematical architecture of the 10-TET topological orbit reveals a profound algebraic elegance not found in the standard 12-TET universe. To systematically translate a complete musical composition between any of these four decaphonic systems, one requires solely a pure linear transformation $f(x) \equiv ax \pmod{10}$. The affine shift parameter $b$, which was strictly necessary in 12-TET to elegantly align the harmonic roots, is entirely superfluous here. By applying these simple scalar multiplications, the structural integrity of the original piece---its voice-leading syntax, harmonic logic, and Tonnetz pathways---is flawlessly transported. Because every multiplier $a \in \mathrm{Aut}(\Z_{10})$ naturally generates an orientation-preserving bijection, the harmony is effortlessly recast into parallel acoustic dimensions, unconditionally retaining its original Major/Minor polarity without the need for any affine recalibration.
\end{conclusion}

\subsection{The Universal Orbit Theorem}

The exhaustive analyses of both the 12-TET and 10-TET landscapes lead us to a profound, unifying mathematical conclusion that binds the abstract topology of Levi graphs to the algebraic reality of pitch classes and the Chinese Remainder Theorem.

\begin{theorem}[Topological and Musical Equivalence]
Let $n \in \{10, 12\}$. Consider the topological orbit of any non-degenerate harmonic system, defined as the set of all systems whose Levi graphs are abstractly isomorphic. This abstract orbit is precisely identical to the set of all systems that are musically equivalent, meaning they are mutually interconnected via affine note-induced isomorphisms $f(x) \equiv ax+b \pmod{n}$. 

Furthermore, when $n$ is a composite of coprime factors (such as $12 = 3 \times 4$ or $10 = 2 \times 5$), this universal orbit always intrinsically contains the foundational canonical system constructed purely from the primitive basis vectors $(1,0)$ and $(0,1)$ of the Chinese Remainder Theorem grid (namely, $(9,4)$ for 12-TET and $(6,5)$ for 10-TET).
\end{theorem}

This theorem definitively bridges abstract graph theory with the tangible acoustics of musical composition. It confirms that whenever a harmonic topology exists, its constituent universes can physically be played on an instrument by applying consistent algebraic transformations to the notes, with the CRT grid acting as the absolute coordinate frame.

\section{A Musician's Guide to Topological Harmony}
\label{sec:musicians_guide}

The mathematical architecture developed in the previous sections provides a rigorous method for discovering new harmonic universes. However, to a composer or an improvising jazz musician, a harmonic system is not a set of algebraic equations---it is a physical terrain of chords, voice-leadings, and progressions. The purpose of this section is to translate our abstract note-induced isomorphisms into a practical, playable guide.

For each system, we will provide the exact chord structures and the permissible Tonnetz pathways. Crucially, when visualising the Levi graphs for the exotic systems, we are faced with a representational choice. Should we preserve the geometric shape of the network, or should we preserve the physical positions of the notes? To fully understand the implications of these parallel universes, we will display \textbf{two parallel layouts} for each exotic system:

\begin{enumerate}
    \item \textbf{The Topological Layout (Isomorphic):} Here, we rearrange the physical position of the chords along the perimeter so that the visual shape of the network (its Hamiltonian cycle) perfectly matches the classical (4,3) graph. Because we strictly utilise the most elegant note transformations derived in Section 3, $M_0$ naturally anchors the top (12 o'clock) position. This layout explicitly proves the topological isomorphism to the eye.
    \item \textbf{The Pitch-Space Layout (Geometric):} Here, we strictly anchor the nodes to their traditional positions on the standard Circle of Fifths. We then draw the new Tonnetz connections between them. This layout reveals the physical reality of playing the system: it shows how the new harmonic rules re-wire our familiar musical space. The edge colours (Red, Blue, Gold) are mapped strictly from their abstract origins.
\end{enumerate}

\subsection{The 12-TET Harmonic Universes: The Dual Pairs}

To make the algebraic framework practically applicable, we must first anchor our numerical system $\Z_{12} = \{0, 1, \dots, 11\}$ to the standard Western chromatic scale:

{\tiny \begin{table}[H]
    \centering
    \begin{tabular}{cccccccccccc}
    \toprule
    \textbf{0} & \textbf{1} & \textbf{2} & \textbf{3} & \textbf{4} & \textbf{5} & \textbf{6} & \textbf{7} & \textbf{8} & \textbf{9} & \textbf{10} & \textbf{11} \\
    \midrule
    C & C$^\sharp$ & D & E$^\flat$ & E & F & F$^\sharp$ & G & A$^\flat$ & A & B$^\flat$ & B \\
    \bottomrule
    \end{tabular}
    \caption{Identification of $\Z_{12}$ elements with classical musical pitch classes.}
\end{table}}

The four structurally viable harmonic systems in 12-TET naturally organise themselves into two distinct pairs, connected by a profound geometric symmetry.

\subsubsection{The Classical (4,3) System}

In the classical Western system, the generator is the perfect fifth $q=7$, split into a Major Third $t=4$ and a Minor Third $s=3$. 

{\tiny \begin{table}[H]
    \centering
    \begin{minipage}{0.48\textwidth}
        \centering
        \begin{tabular}{ll}
        \toprule
        \textbf{Classical Name} & \textbf{Numerical ($M_r$)} \\
        \midrule
        C = \{C, E, G\} & $M_0 = \{0, 4, 7\}$ \\
        C$^\sharp$ = \{C$^\sharp$, F, G$^\sharp$\} & $M_1 = \{1, 5, 8\}$ \\
        D = \{D, F$^\sharp$, A\} & $M_2 = \{2, 6, 9\}$ \\
        E$^\flat$ = \{E$^\flat$, G, B$^\flat$\} & $M_3 = \{3, 7, 10\}$ \\
        E = \{E, G$^\sharp$, B\} & $M_4 = \{4, 8, 11\}$ \\
        F = \{F, A, C\} & $M_5 = \{5, 9, 0\}$ \\
        F$^\sharp$ = \{F$^\sharp$, A$^\sharp$, C$^\sharp$\} & $M_6 = \{6, 10, 1\}$ \\
        G = \{G, B, D\} & $M_7 = \{7, 11, 2\}$ \\
        A$^\flat$ = \{A$^\flat$, C, E$^\flat$\} & $M_8 = \{8, 0, 3\}$ \\
        A = \{A, C$^\sharp$, E\} & $M_9 = \{9, 1, 4\}$ \\
        B$^\flat$ = \{B$^\flat$, D, F\} & $M_{10} = \{10, 2, 5\}$ \\
        B = \{B, D$^\sharp$, F$^\sharp$\} & $M_{11} = \{11, 3, 6\}$ \\
        \bottomrule
        \end{tabular}
    \end{minipage}\hfill
    \begin{minipage}{0.48\textwidth}
        \centering
        \begin{tabular}{ll}
        \toprule
        \textbf{Classical Name} & \textbf{Numerical ($m_r$)} \\
        \midrule
        Cm = \{C, E$^\flat$, G\} & $m_0 = \{0, 3, 7\}$ \\
        C$^\sharp$m = \{C$^\sharp$, E, G$^\sharp$\} & $m_1 = \{1, 4, 8\}$ \\
        Dm = \{D, F, A\} & $m_2 = \{2, 5, 9\}$ \\
        E$^\flat$m = \{E$^\flat$, F$^\sharp$, B$^\flat$\} & $m_3 = \{3, 6, 10\}$ \\
        Em = \{E, G, B\} & $m_4 = \{4, 7, 11\}$ \\
        Fm = \{F, A$^\flat$, C\} & $m_5 = \{5, 8, 0\}$ \\
        F$^\sharp$m = \{F$^\sharp$, A, C$^\sharp$\} & $m_6 = \{6, 9, 1\}$ \\
        Gm = \{G, B$^\flat$, D\} & $m_7 = \{7, 10, 2\}$ \\
        G$^\sharp$m = \{A$^\flat$, B, E$^\flat$\} & $m_8 = \{8, 11, 3\}$ \\
        Am = \{A, C, E\} & $m_9 = \{9, 0, 4\}$ \\
        B$^\flat$m = \{B$^\flat$, C$^\sharp$, F\} & $m_{10} = \{10, 1, 5\}$ \\
        Bm = \{B, D, F$^\sharp$\} & $m_{11} = \{11, 2, 6\}$ \\
        \bottomrule
        \end{tabular}
    \end{minipage}
    \caption{The 24 Major and Minor Chords in the classical (4,3) system.}
\end{table}}

The Tonnetz dictates specific harmonic progressions. A Major chord $M_r$ connects to $m_r, m_{r+4}$, and $m_{r-3=r+9}$. A Minor chord $m_r$ connects to $M_r, M_{r-4=r+8}$, and $M_{r+3}$. 

{\tiny \begin{table}[H]
    \centering
    \begin{minipage}{0.48\textwidth}
        \centering
        \begin{tabular}{ll}
        \toprule
        \textbf{Major $\to$ Minor} & \textbf{Numerical} \\
        \midrule
        C $\to$ Cm, Em, Am & $M_0 \to m_0, m_4, m_9$ \\
        C$^\sharp$ $\to$ C$^\sharp$m, Fm, B$^\flat$m & $M_1 \to m_1, m_5, m_{10}$ \\
        D $\to$ Dm, F$^\sharp$m, Bm & $M_2 \to m_2, m_6, m_{11}$ \\
        E$^\flat$ $\to$ E$^\flat$m, Gm, Cm & $M_3 \to m_3, m_7, m_0$ \\
        E $\to$ Em, G$^\sharp$m, C$^\sharp$m & $M_4 \to m_4, m_8, m_1$ \\
        F $\to$ Fm, Am, Dm & $M_5 \to m_5, m_9, m_2$ \\
        F$^\sharp$ $\to$ F$^\sharp$m, B$^\flat$m, E$^\flat$m & $M_6 \to m_6, m_{10}, m_3$ \\
        G $\to$ Gm, Bm, Em & $M_7 \to m_7, m_{11}, m_4$ \\
        A$^\flat$ $\to$ G$^\sharp$m, Cm, Fm & $M_8 \to m_8, m_0, m_5$ \\
        A $\to$ Am, C$^\sharp$m, F$^\sharp$m & $M_9 \to m_9, m_1, m_6$ \\
        B$^\flat$ $\to$ B$^\flat$m, Dm, Gm & $M_{10} \to m_{10}, m_2, m_7$ \\
        B $\to$ Bm, E$^\flat$m, G$^\sharp$m & $M_{11} \to m_{11}, m_3, m_8$ \\
        \bottomrule
        \end{tabular}
    \end{minipage}\hfill
    \begin{minipage}{0.48\textwidth}
        \centering
        \begin{tabular}{ll}
        \toprule
        \textbf{Minor $\to$ Major} & \textbf{Numerical} \\
        \midrule
        Cm $\to$ C, A$^\flat$, E$^\flat$ & $m_0 \to M_0, M_8, M_3$ \\
        C$^\sharp$m $\to$ C$^\sharp$, A, E & $m_1 \to M_1, M_9, M_4$ \\
        Dm $\to$ D, B$^\flat$, F & $m_2 \to M_2, M_{10}, M_5$ \\
        E$^\flat$m $\to$ E$^\flat$, B, F$^\sharp$ & $m_3 \to M_3, M_{11}, M_6$ \\
        Em $\to$ E, C, G & $m_4 \to M_4, M_0, M_7$ \\
        Fm $\to$ F, C$^\sharp$, A$^\flat$ & $m_5 \to M_5, M_1, M_8$ \\
        F$^\sharp$m $\to$ F$^\sharp$, D, A & $m_6 \to M_6, M_2, M_9$ \\
        Gm $\to$ G, E$^\flat$, B$^\flat$ & $m_7 \to M_7, M_3, M_{10}$ \\
        G$^\sharp$m $\to$ A$^\flat$, E, B & $m_8 \to M_8, M_4, M_{11}$ \\
        Am $\to$ A, F, C & $m_9 \to M_9, M_5, M_0$ \\
        B$^\flat$m $\to$ B$^\flat$, F$^\sharp$, C$^\sharp$ & $m_{10} \to M_{10}, M_6, M_1$ \\
        Bm $\to$ B, G, D & $m_{11} \to M_{11}, M_7, M_2$ \\
        \bottomrule
        \end{tabular}
    \end{minipage}
    \caption{Tonnetz progression map for the classical (4,3) system.}
\end{table}}

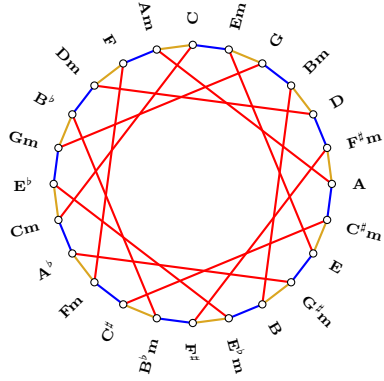
\begin{figure}[H]
    \centering
    \begin{tikzpicture}[scale=0.8, transform shape]
        \tikzset{snode/.style={circle, draw=black, fill=white, inner sep=1.2pt}}
        \def\n{12} \def\R{2.3} \def\Rtext{2.5} 
        
        \def\names{{"C","C$^\sharp$","D","E$^\flat$","E","F","F$^\sharp$","G","A$^\flat$","A","B$^\flat$","B"}}
        \def\minnames{{"Cm","C$^\sharp$m","Dm","E$^\flat$m","Em","Fm","F$^\sharp$m","Gm","G$^\sharp$m","Am","B$^\flat$m","Bm"}}

        \foreach \s in {0,...,11} {
            \pgfmathsetmacro{\k}{int(mod(5*\s, 12))}
            \pgfmathsetmacro{\j}{int(mod(9 + 5*\s, 12))}
            \pgfmathsetmacro{\angM}{90 + (\s*2)*15}
            \pgfmathsetmacro{\angm}{90 + (\s*2 + 1)*15}
            \coordinate (M\k) at (\angM:\R);
            \coordinate (m\j) at (\angm:\R);
        }
        
        \foreach \i in {0,...,11} {
            \draw[red, thick] (M\i) -- (m\i);
            \pgfmathsetmacro{\idxL}{int(mod(\i+4, \n))}
            \draw[blue, thick] (M\i) -- (m\idxL);
            \pgfmathsetmacro{\idxR}{int(mod(\i+9, \n))}
            \draw[levigold, thick] (M\i) -- (m\idxR);
        }
        
        \foreach \s in {0,...,11} {
            \pgfmathsetmacro{\k}{int(mod(5*\s, 12))}
            \pgfmathsetmacro{\j}{int(mod(9 + 5*\s, 12))}
            \pgfmathsetmacro{\angM}{90 + (\s*2)*15}
            \pgfmathsetmacro{\angm}{90 + (\s*2 + 1)*15}
            \node[snode] at (M\k) {};
            \node[snode] at (m\j) {};
            
            \pgfmathparse{\names[\k]} \let\currM\pgfmathresult
            \pgfmathparse{\minnames[\j]} \let\currm\pgfmathresult

            \pgfmathsetmacro{\normM}{mod(\angM+360, 360)}
            \pgfmathsetmacro{\normm}{mod(\angm+360, 360)}
            \pgfmathsetmacro{\rotM}{ (\normM > 90.1 && \normM < 269.9) ? \angM - 180 : \angM }
            \pgfmathsetmacro{\rotm}{ (\normm > 90.1 && \normm < 269.9) ? \angm - 180 : \angm }
            
            \pgfmathsetmacro{\leftM}{ (\normM > 90.1 && \normM < 269.9) ? 1 : 0 }
            \ifnum\leftM=1
                \node[font=\fontsize{7}{8}\selectfont, anchor=east, rotate=\rotM] at (\angM:\Rtext) {\textbf{\currM}};
            \else
                \node[font=\fontsize{7}{8}\selectfont, anchor=west, rotate=\rotM] at (\angM:\Rtext) {\textbf{\currM}};
            \fi

            \pgfmathsetmacro{\leftm}{ (\normm > 90.1 && \normm < 269.9) ? 1 : 0 }
            \ifnum\leftm=1
                \node[font=\fontsize{7}{8}\selectfont, anchor=east, rotate=\rotm] at (\angm:\Rtext) {\textbf{\currm}};
            \else
                \node[font=\fontsize{7}{8}\selectfont, anchor=west, rotate=\rotm] at (\angm:\Rtext) {\textbf{\currm}};
            \fi
        }
    \end{tikzpicture}
    \caption{A Levi graph mapping the standard Western (4,3) Tonnetz. Connecting lines indicate voice-leading: Parallel (Red), Leading-Tone (Blue), and Relative (Gold).}
    \label{fig:levi_standard}
\end{figure}

\subsubsection{System (9,8): The Mirror Universe of the Classical Standard}

In our first exotic isomorphism, the generator is a perfect fourth ($q=5$). Because this is generated by the orientation-reversing mapping $f(x)=11x \equiv -x$, it flips the orientation of the pitch space.

{\tiny \begin{table}[H]
    \centering
    \begin{minipage}{0.48\textwidth}
        \centering
        \begin{tabular}{l}
        \toprule
        \textbf{Numerical ($M_r$)} \\
        \midrule
        $M_0 = \{\text{C, A, F}\} = \{0, 9, 5\}$ \\
        $M_1 = \{\text{C}^\sharp, \text{B}^\flat, \text{F}^\sharp\} = \{1, 10, 6\}$ \\
        $M_2 = \{\text{D, B, G}\} = \{2, 11, 7\}$ \\
        $M_3 = \{\text{E}^\flat, \text{C, A}^\flat\} = \{3, 0, 8\}$ \\
        $M_4 = \{\text{E, C}^\sharp, \text{A}\} = \{4, 1, 9\}$ \\
        $M_5 = \{\text{F, D, B}^\flat\} = \{5, 2, 10\}$ \\
        $M_6 = \{\text{F}^\sharp, \text{E}^\flat, \text{B}\} = \{6, 3, 11\}$ \\
        $M_7 = \{\text{G, E, C}\} = \{7, 4, 0\}$ \\
        $M_8 = \{\text{A}^\flat, \text{F, C}^\sharp\} = \{8, 5, 1\}$ \\
        $M_9 = \{\text{A, F}^\sharp, \text{D}\} = \{9, 6, 2\}$ \\
        $M_{10} = \{\text{B}^\flat, \text{G, E}^\flat\} = \{10, 7, 3\}$ \\
        $M_{11} = \{\text{B, A}^\flat, \text{E}\} = \{11, 8, 4\}$ \\
        \bottomrule
        \end{tabular}
    \end{minipage}\hfill
    \begin{minipage}{0.48\textwidth}
        \centering
        \begin{tabular}{l}
        \toprule
        \textbf{Numerical ($m_r$)} \\
        \midrule
        $m_0 = \{\text{C, A}^\flat, \text{F}\} = \{0, 8, 5\}$ \\
        $m_1 = \{\text{C}^\sharp, \text{A, F}^\sharp\} = \{1, 9, 6\}$ \\
        $m_2 = \{\text{D, B}^\flat, \text{G}\} = \{2, 10, 7\}$ \\
        $m_3 = \{\text{E}^\flat, \text{B, A}^\flat\} = \{3, 11, 8\}$ \\
        $m_4 = \{\text{E, C, A}\} = \{4, 0, 9\}$ \\
        $m_5 = \{\text{F, C}^\sharp, \text{B}^\flat\} = \{5, 1, 10\}$ \\
        $m_6 = \{\text{F}^\sharp, \text{D, B}\} = \{6, 2, 11\}$ \\
        $m_7 = \{\text{G, E}^\flat, \text{C}\} = \{7, 3, 0\}$ \\
        $m_8 = \{\text{A}^\flat, \text{E, C}^\sharp\} = \{8, 4, 1\}$ \\
        $m_9 = \{\text{A, F, D}\} = \{9, 5, 2\}$ \\
        $m_{10} = \{\text{B}^\flat, \text{F}^\sharp, \text{E}^\flat\} = \{10, 6, 3\}$ \\
        $m_{11} = \{\text{B, G, E}\} = \{11, 7, 4\}$ \\
        \bottomrule
        \end{tabular}
    \end{minipage}
    \caption{Chords in the (9,8) system.}
\end{table}}

As anticipated by the orientation-reversing isomorphism mathematically proven in Theorem 2, a striking musical equivalence emerges. If we ignore the internal order of the notes, the 24 chords of the (9,8) system are the exact same sets of pitch classes as those in the classical (4,3) system. For example, $M_0 = \{0, 9, 5\}$ contains the exact same notes as the classical F Major triad $\{5, 9, 0\}$. Because the base tone is 0 (C), this is physically an F Major chord played with its fifth in the bass. In music theory, these are known as \textit{second inversion} chords. 

{\tiny \begin{table}[H]
    \centering
    \begin{minipage}{0.48\textwidth}
        \centering
        \begin{tabular}{ll}
        \toprule
        \textbf{Classical (2nd Inversion)} & \textbf{Numerical ($M_r$)} \\
        \midrule
        F/C & $M_0 = \{0, 9, 5\}$ \\
        F$^\sharp$/C$^\sharp$ & $M_1 = \{1, 10, 6\}$ \\
        G/D & $M_2 = \{2, 11, 7\}$ \\
        A$^\flat$/E$^\flat$ & $M_3 = \{3, 0, 8\}$ \\
        A/E & $M_4 = \{4, 1, 9\}$ \\
        B$^\flat$/F & $M_5 = \{5, 2, 10\}$ \\
        B/F$^\sharp$ & $M_6 = \{6, 3, 11\}$ \\
        C/G & $M_7 = \{7, 4, 0\}$ \\
        C$^\sharp$/G$^\sharp$ & $M_8 = \{8, 5, 1\}$ \\
        D/A & $M_9 = \{9, 6, 2\}$ \\
        E$^\flat$/B$^\flat$ & $M_{10} = \{10, 7, 3\}$ \\
        E/B & $M_{11} = \{11, 8, 4\}$ \\
        \bottomrule
        \end{tabular}
    \end{minipage}\hfill
    \begin{minipage}{0.48\textwidth}
        \centering
        \begin{tabular}{ll}
        \toprule
        \textbf{Classical (2nd Inversion)} & \textbf{Numerical ($m_r$)} \\
        \midrule
        Fm/C & $m_0 = \{0, 8, 5\}$ \\
        F$^\sharp$m/C$^\sharp$ & $m_1 = \{1, 9, 6\}$ \\
        Gm/D & $m_2 = \{2, 10, 7\}$ \\
        A$^\flat$m/E$^\flat$ & $m_3 = \{3, 11, 8\}$ \\
        Am/E & $m_4 = \{4, 0, 9\}$ \\
        B$^\flat$m/F & $m_5 = \{5, 1, 10\}$ \\
        Bm/F$^\sharp$ & $m_6 = \{6, 2, 11\}$ \\
        Cm/G & $m_7 = \{7, 3, 0\}$ \\
        C$^\sharp$m/G$^\sharp$ & $m_8 = \{8, 4, 1\}$ \\
        Dm/A & $m_9 = \{9, 5, 2\}$ \\
        E$^\flat$m/B$^\flat$ & $m_{10} = \{10, 6, 3\}$ \\
        Em/B & $m_{11} = \{11, 7, 4\}$ \\
        \bottomrule
        \end{tabular}
    \end{minipage}
    \caption{The physical equivalence of (9,8) chords to standard Western second inversions.}
\end{table}}

With the physical identity of these chords established, we now turn to their voice-leading pathways. The permissible harmonic progressions that weave these chords together across the (9,8) Tonnetz are mapped out below.

{\tiny \begin{table}[H]
    \centering
    \begin{minipage}{0.48\textwidth}
        \centering
        \begin{tabular}{l}
        \toprule
        \textbf{Major $\to$ Minor} \\
        \midrule
        $M_0 \to m_0, m_9, m_4$ \\
        $M_1 \to m_1, m_{10}, m_5$ \\
        $M_2 \to m_2, m_{11}, m_6$ \\
        $M_3 \to m_3, m_0, m_7$ \\
        $M_4 \to m_4, m_1, m_8$ \\
        $M_5 \to m_5, m_2, m_9$ \\
        $M_6 \to m_6, m_3, m_{10}$ \\
        $M_7 \to m_7, m_4, m_{11}$ \\
        $M_8 \to m_8, m_5, m_0$ \\
        $M_9 \to m_9, m_6, m_1$ \\
        $M_{10} \to m_{10}, m_7, m_2$ \\
        $M_{11} \to m_{11}, m_8, m_3$ \\
        \bottomrule
        \end{tabular}
    \end{minipage}\hfill
    \begin{minipage}{0.48\textwidth}
        \centering
        \begin{tabular}{l}
        \toprule
        \textbf{Minor $\to$ Major} \\
        \midrule
        $m_0 \to M_0, M_3, M_8$ \\
        $m_1 \to M_1, M_4, M_9$ \\
        $m_2 \to M_2, M_5, M_{10}$ \\
        $m_3 \to M_3, M_6, M_{11}$ \\
        $m_4 \to M_4, M_7, M_0$ \\
        $m_5 \to M_5, M_8, M_1$ \\
        $m_6 \to M_6, M_9, M_2$ \\
        $m_7 \to M_7, M_{10}, M_3$ \\
        $m_8 \to M_8, M_{11}, M_4$ \\
        $m_9 \to M_9, M_0, M_5$ \\
        $m_{10} \to M_{10}, M_1, M_6$ \\
        $m_{11} \to M_{11}, M_2, M_7$ \\
        \bottomrule
        \end{tabular}
    \end{minipage}
    \caption{Tonnetz progression map for the (9,8) system.}
\end{table}}

As dictated by the note-induced isomorphisms from Section 3, these Tonnetz connections reveal another perfect equivalence. In the classical (4,3) system, $M_r$ connects to $m_r, m_{r+4}$, and $m_{r+9}$. In this (9,8) system, $M_r$ connects to exactly the same destinations: $m_r, m_{r+9}$, and $m_{r+4}$. The permissible connections are identical; only the $L$ and $R$ voice-leading directions ($+4$ and $+9$) have been perfectly swapped, explicitly confirming the chiral reflection.

These exact equivalences in both chord structure and voice-leading paths are geometrically visualised in Figure \ref{fig:negative_harmony}.
\vspace{0.4cm}

\noindent\textbf{Navigating the Mirror Universe: Interpreting the (9,8) Diagrams}

Before examining the network of the $(9,8)$ system, it is vital to establish how a performing musician should interpret its dual representations. Because the $24$ triads of the $(9,8)$ system are element-wise identical as sets of pitch classes to classical Western triads (manifesting physically as standard major and minor chords in second inversion), this system occupies a deeply familiar terrain within a reversed structural framework:

\begin{itemize}
    \item \textbf{The Topological Layout (left graph):} Serves as a formal analytical baseline. It reorders the positions of the chords along the perimeter to match the pristine Hamiltonian cycle of the classical $(4,3)$ Tonnetz, with $M_0$ anchored at 12 o'clock. This explicitly demonstrates that the underlying voice-leading syntax remains perfectly intact under the transformation.
    \item \textbf{The Circle of Fifths Layout (right graph):} Functions as a direct, practical roadmap for improvisation and harmonic navigation. Since the musician is performing chords whose acoustic and note-content identities are entirely familiar, this layout freezes the triads on their traditional, intuitive positions on the standard Circle of Fifths. The rewired connecting lines reveal exactly how Negative Harmony structurally inverts the pathways between these familiar chords, showing how traditional musical geography is transformed in real-time performance.
\end{itemize}
\vspace{0.4cm}

\begin{figure}[H]
    \centering
    \newcommand{\getChordM}[1]{%
        \ifcase#1 \mathrm{F}/\mathrm{C}\or \mathrm{F}^\sharp/\mathrm{C}^\sharp\or \mathrm{G}/\mathrm{D}\or \mathrm{A}^\flat/\mathrm{E}^\flat\or \mathrm{A}/\mathrm{E}\or \mathrm{B}^\flat/\mathrm{F}\or \mathrm{B}/\mathrm{F}^\sharp\or \mathrm{C}/\mathrm{G}\or \mathrm{C}^\sharp/\mathrm{G}^\sharp\or \mathrm{D}/\mathrm{A}\or \mathrm{E}^\flat/\mathrm{B}^\flat\or \mathrm{E}/\mathrm{B}\fi
    }
    \newcommand{\getChordm}[1]{%
        \ifcase#1 \mathrm{Fm}/\mathrm{C}\or \mathrm{F}^\sharp\mathrm{m}/\mathrm{C}^\sharp\or \mathrm{Gm}/\mathrm{D}\or \mathrm{A}^\flat\mathrm{m}/\mathrm{E}^\flat\or \mathrm{Am}/\mathrm{E}\or \mathrm{B}^\flat\mathrm{m}/\mathrm{F}\or \mathrm{Bm}/\mathrm{F}^\sharp\or \mathrm{Cm}/\mathrm{G}\or \mathrm{C}^\sharp\mathrm{m}/\mathrm{G}^\sharp\or \mathrm{Dm}/\mathrm{A}\or \mathrm{E}^\flat\mathrm{m}/\mathrm{B}^\flat\or \mathrm{Em}/\mathrm{B}\fi
    }
    
    \begin{minipage}{0.48\textwidth}
        \centering
        \begin{tikzpicture}[scale=0.7, transform shape]
            \tikzset{snode/.style={circle, draw=black, fill=white, inner sep=1.2pt}}
            \def\n{12} \def\R{2.0} \def\Rtext{2.5} 
            
            \foreach \s in {0,...,11} {
                \pgfmathsetmacro{\k}{int(mod(5*\s, 12))}
                \pgfmathsetmacro{\j}{int(mod(9 + 5*\s, 12))}
                \pgfmathsetmacro{\angM}{90 + (\s*2)*15 - 105}
                \pgfmathsetmacro{\angm}{90 + (\s*2 + 1)*15 - 105}
                \coordinate (M\k) at (\angM:\R);
                \coordinate (m\j) at (\angm:\R);
            }
            \foreach \i in {0,...,11} {
                \draw[red, thick] (M\i) -- (m\i);
                \pgfmathsetmacro{\idxL}{int(mod(\i+4, \n))}
                \draw[blue, thick] (M\i) -- (m\idxL);
                \pgfmathsetmacro{\idxR}{int(mod(\i+9, \n))}
                \draw[levigold, thick] (M\i) -- (m\idxR);
            }
            \foreach \s in {0,...,11} {
                \pgfmathsetmacro{\k}{int(mod(5*\s, 12))}
                \pgfmathsetmacro{\j}{int(mod(9 + 5*\s, 12))}
                \pgfmathsetmacro{\angM}{90 + (\s*2)*15 - 105}
                \pgfmathsetmacro{\angm}{90 + (\s*2 + 1)*15 - 105}
          
                \node[snode] at (M\k) {};
                \node[snode] at (m\j) {};
                
                \pgfmathsetmacro{\labelK}{int(mod(11*\k, 12))}
                \pgfmathsetmacro{\labelJ}{int(mod(11*\j, 12))}
                
                \pgfmathsetmacro{\normM}{mod(\angM+360, 360)}
                \pgfmathsetmacro{\normm}{mod(\angm+360, 360)}
                \pgfmathsetmacro{\rotM}{ (\normM > 90.1 && \normM < 269.9) ? \angM - 180 : \angM }
                \pgfmathsetmacro{\rotm}{ (\normm > 90.1 && \normm < 269.9) ? \angm - 180 : \angm }
                
                \pgfmathsetmacro{\leftM}{ (\normM > 90.1 && \normM < 269.9) ? 1 : 0 }
                \ifnum\leftM=1
                    \node[font=\fontsize{7}{8}\selectfont, anchor=east, rotate=\rotM] at (\angM:\Rtext) {$\getChordm{\labelK}$};
                \else
                    \node[font=\fontsize{7}{8}\selectfont, anchor=west, rotate=\rotM] at (\angM:\Rtext) {$\getChordm{\labelK}$};
                \fi

                \pgfmathsetmacro{\leftm}{ (\normm > 90.1 && \normm < 269.9) ? 1 : 0 }
                \ifnum\leftm=1
                    \node[font=\fontsize{7}{8}\selectfont, anchor=east, rotate=\rotm] at (\angm:\Rtext) {$\getChordM{\labelJ}$};
                \else
                    \node[font=\fontsize{7}{8}\selectfont, anchor=west, rotate=\rotm] at (\angm:\Rtext) {$\getChordM{\labelJ}$};
                \fi
            }
        \end{tikzpicture}
        \caption*{Topological Layout \\ (Proof of Isomorphism)}
    \end{minipage}\hfill
    \begin{minipage}{0.48\textwidth}
        \centering
        \begin{tikzpicture}[scale=0.7, transform shape]
            \tikzset{snode/.style={circle, draw=black, fill=white, inner sep=1.2pt}}
            \def\n{12} \def\R{2.0} \def\Rtext{2.5} 
          
            \foreach \s in {0,...,11} {
                \pgfmathsetmacro{\k}{int(mod(5*\s, 12))}
                \pgfmathsetmacro{\j}{int(mod(9 + 5*\s, 12))}
                \pgfmathsetmacro{\angM}{90 + (\s*2)*15}
                \pgfmathsetmacro{\angm}{90 + (\s*2 + 1)*15}
        
                \coordinate (M\k) at (\angM:\R);
                \coordinate (m\j) at (\angm:\R);
            }
            \foreach \i in {0,...,11} {
                \draw[red, thick] (M\i) -- (m\i);
                \pgfmathsetmacro{\idxL}{int(mod(\i+4, \n))} 
                \draw[blue, thick] (M\i) -- (m\idxL);
                \pgfmathsetmacro{\idxR}{int(mod(\i+9, \n))}
                \draw[levigold, thick] (M\i) -- (m\idxR);
            }
            \foreach \s in {0,...,11} {
                \pgfmathsetmacro{\k}{int(mod(5*\s, 12))}
                \pgfmathsetmacro{\j}{int(mod(9 + 5*\s, 12))}
                \pgfmathsetmacro{\angM}{90 + (\s*2)*15}
                \pgfmathsetmacro{\angm}{90 + (\s*2 + 1)*15}
          
                \node[snode] at (M\k) {};
                \node[snode] at (m\j) {};
                
                \pgfmathsetmacro{\normM}{mod(\angM+360, 360)}
                \pgfmathsetmacro{\normm}{mod(\angm+360, 360)}
                \pgfmathsetmacro{\rotM}{ (\normM > 90.1 && \normM < 269.9) ? \angM - 180 : \angM }
                \pgfmathsetmacro{\rotm}{ (\normm > 90.1 && \normm < 269.9) ? \angm - 180 : \angm }
                
                \pgfmathsetmacro{\leftM}{ (\normM > 90.1 && \normM < 269.9) ? 1 : 0 }
                \ifnum\leftM=1
                    \node[font=\fontsize{7}{8}\selectfont, anchor=east, rotate=\rotM] at (\angM:\Rtext) {$\getChordM{\k}$};
                \else
                    \node[font=\fontsize{7}{8}\selectfont, anchor=west, rotate=\rotM] at (\angM:\Rtext) {$\getChordM{\k}$};
                \fi

                \pgfmathsetmacro{\leftm}{ (\normm > 90.1 && \normm < 269.9) ? 1 : 0 }
                \ifnum\leftm=1
                    \node[font=\fontsize{7}{8}\selectfont, anchor=east, rotate=\rotm] at (\angm:\Rtext) {$\getChordm{\j}$};
                \else
                    \node[font=\fontsize{7}{8}\selectfont, anchor=west, rotate=\rotm] at (\angm:\Rtext) {$\getChordm{\j}$};
                \fi
            }
        \end{tikzpicture}
        \caption*{Pitch-Space Layout \\ (Geometric Rewiring)}
    \end{minipage}
    \caption{Two views of the (9,8) system. Notice that the Pitch-Space geometry (right) is identical to the standard (4,3) system, visually confirming that the (9,8) universe is physically the exact same musical object, merely traversed in opposite voice-leading directions (+4 vs +9).}
    \label{fig:negative_harmony}
\end{figure}

As visually evident in Figure \ref{fig:negative_harmony}, the Pitch-Space layout for the $(9,8)$ system exhibits a remarkable geometric property. The physical network remains a perfect, orderly circle of fifths. This visual anomaly is the geometric manifestation of the harmonic dualism proven in Section 3. As we have established, the affine multiplier $11 \equiv -1 \pmod{12}$ makes the $(9,8)$ harmonic system the exact pitch-inversion of the classical $(4,3)$ system. 

\begin{remark}[Negative Harmony]
  Because the note bijection $f(x) \equiv -x + b$ simply reverses the orientation of the pitch-class circle, the relative distances between notes are preserved. Musically, a major chord built upwards in the standard $(4,3)$ system becomes a chord built downwards in the $(9,8)$ system. This symmetry serves as the rigorous mathematical foundation for \textit{Negative Harmony} (as historically proposed by Ernst Levy \cite{levy_harmony}).
\end{remark}

Furthermore, as a direct corollary of this orientation-reversing isomorphism, the $(9,8)$ system utilises the exact same 24 triads as the classical $(4,3)$ system. Specifically:
\begin{equation}
    M^{(9,8)}_k = M^{(4,3)}_{k+5} \quad \text{and} \quad m^{(9,8)}_k = m^{(4,3)}_{k+5}.
\end{equation}
Playing the $(9,8)$ system is physically equivalent to playing standard Western Major and Minor chords in their \textit{second inversion}, whilst treating that classical fifth as the new root.

\begin{conclusion}[The Musician's Playbook: (4,3) vs (9,8)]
For the performing musician, the theoretical duality between the $(4,3)$ and $(9,8)$ systems translates into a surprisingly simple practical rule. To play a composition in the exotic $(9,8)$ universe, \textbf{you do not need to learn any new chords}. You simply play traditional major and minor chords in their second inversion (with the classical fifth in the bass acting as the new root). 

The true performative challenge lies solely in muscle memory regarding progressions. Because the system is an exact chiral reflection (Negative Harmony), your internal voice-leading compass must be inverted. When your musical intuition tells your hands to execute the physical finger movements for a `Leading-Tone' shift (moving the root by a semitone), the $(9,8)$ geometry dictates that you must instead physically execute the `Relative' shift, and vice versa. It is playing in a familiar physical space, but navigating with a mirrored map.
\end{conclusion}

\subsubsection{The Fundamental Anchor: System (9,4)}

We now turn to our second dual pair. In the foundational $(9,4)$ system, the generator is the simple semitone $q = 1$. Because this system is generated by the highly elegant orientation-preserving transformation $f(x) = 5x+1$, Major chords rigorously map to Major chords. The explicit pitch-class contents of these exotic triads are structured as follows:

{\tiny \begin{table}[H]
    \centering
    \begin{minipage}{0.48\textwidth}
        \centering
        \begin{tabular}{l}
        \toprule
        \textbf{Numerical ($M_r$)} \\
        \midrule
        $M_0 = \{\text{C, A, C}^\sharp\} = \{0, 9, 1\}$ \\
        $M_1 = \{\text{C}^\sharp, \text{B}^\flat, \text{D}\} = \{1, 10, 2\}$ \\
        $M_2 = \{\text{D, B, E}^\flat\} = \{2, 11, 3\}$ \\
        $M_3 = \{\text{E}^\flat, \text{C, E}\} = \{3, 0, 4\}$ \\
        $M_4 = \{\text{E, C}^\sharp, \text{F}\} = \{4, 1, 5\}$ \\
        $M_5 = \{\text{F, D, F}^\sharp\} = \{5, 2, 6\}$ \\
        $M_6 = \{\text{F}^\sharp, \text{E}^\flat, \text{G}\} = \{6, 3, 7\}$ \\
        $M_7 = \{\text{G, E, A}^\flat\} = \{7, 4, 8\}$ \\
        $M_8 = \{\text{A}^\flat, \text{F, A}\} = \{8, 5, 9\}$ \\
        $M_9 = \{\text{A, F}^\sharp, \text{B}^\flat\} = \{9, 6, 10\}$ \\
        $M_{10} = \{\text{B}^\flat, \text{G, B}\} = \{10, 7, 11\}$ \\
        $M_{11} = \{\text{B, A}^\flat, \text{C}\} = \{11, 8, 0\}$ \\
        \bottomrule
        \end{tabular}
    \end{minipage}\hfill
    \begin{minipage}{0.48\textwidth}
        \centering
        \begin{tabular}{l}
        \toprule
        \textbf{Numerical ($m_r$)} \\
        \midrule
        $m_0 = \{\text{C, E, C}^\sharp\} = \{0, 4, 1\}$ \\
        $m_1 = \{\text{C}^\sharp, \text{F, D}\} = \{1, 5, 2\}$ \\
        $m_2 = \{\text{D, F}^\sharp, \text{E}^\flat\} = \{2, 6, 3\}$ \\
        $m_3 = \{\text{E}^\flat, \text{G, E}\} = \{3, 7, 4\}$ \\
        $m_4 = \{\text{E, A}^\flat, \text{F}\} = \{4, 8, 5\}$ \\
        $m_5 = \{\text{F, A, F}^\sharp\} = \{5, 9, 6\}$ \\
        $m_6 = \{\text{F}^\sharp, \text{B}^\flat, \text{G}\} = \{6, 10, 7\}$ \\
        $m_7 = \{\text{G, B, A}^\flat\} = \{7, 11, 8\}$ \\
        $m_8 = \{\text{A}^\flat, \text{C, A}\} = \{8, 0, 9\}$ \\
        $m_9 = \{\text{A, C}^\sharp, \text{B}^\flat\} = \{9, 1, 10\}$ \\
        $m_{10} = \{\text{B}^\flat, \text{D, B}\} = \{10, 2, 11\}$ \\
        $m_{11} = \{\text{B, E}^\flat, \text{C}\} = \{11, 3, 0\}$ \\
        \bottomrule
        \end{tabular}
    \end{minipage}
    \caption{Chords in the (9,4) system.}
\end{table}}

Having defined the internal composition of these exotic chords, we must now determine how they flow into one another. Just as in the classical system, the geometric rules of the Tonnetz dictate that every major chord $M_r$ gracefully transitions into three specific minor chords. These local neighbourhood connections are explicitly calculated below.

{\tiny \begin{table}[H]
    \centering
    \begin{minipage}{0.48\textwidth}
        \centering
        \begin{tabular}{l}
        \toprule
        \textbf{Major $\to$ Minor} \\
        \midrule
        $M_0 \to m_0, m_9, m_8$ \\
        $M_1 \to m_1, m_{10}, m_9$ \\
        $M_2 \to m_2, m_{11}, m_{10}$ \\
        $M_3 \to m_3, m_0, m_{11}$ \\
        $M_4 \to m_4, m_1, m_0$ \\
        $M_5 \to m_5, m_2, m_1$ \\
        $M_6 \to m_6, m_3, m_2$ \\
        $M_7 \to m_7, m_4, m_3$ \\
        $M_8 \to m_8, m_5, m_4$ \\
        $M_9 \to m_9, m_6, m_5$ \\
        $M_{10} \to m_{10}, m_7, m_6$ \\
        $M_{11} \to m_{11}, m_8, m_7$ \\
        \bottomrule
        \end{tabular}
    \end{minipage}\hfill
    \begin{minipage}{0.48\textwidth}
        \centering
        \begin{tabular}{l}
        \toprule
        \textbf{Minor $\to$ Major} \\
        \midrule
        $m_0 \to M_0, M_3, M_4$ \\
        $m_1 \to M_1, M_4, M_5$ \\
        $m_2 \to M_2, M_5, M_6$ \\
        $m_3 \to M_3, M_6, M_7$ \\
        $m_4 \to M_4, M_7, M_8$ \\
        $m_5 \to M_5, M_8, M_9$ \\
        $m_6 \to M_6, M_9, M_{10}$ \\
        $m_7 \to M_7, M_{10}, M_{11}$ \\
        $m_8 \to M_8, M_{11}, M_0$ \\
        $m_9 \to M_9, M_0, M_1$ \\
        $m_{10} \to M_{10}, M_1, M_2$ \\
        $m_{11} \to M_{11}, M_2, M_3$ \\
        \bottomrule
        \end{tabular}
    \end{minipage}
    \caption{Tonnetz progression map for the (9,4) system.}
\end{table}}

Because the generating map $f(x)=5x+1$ is perfectly orientation-preserving, the topological layout of the $(9,4)$ system maintains the exact same pristine orientation as the classical $(4,3)$ system. The nodes remain exactly where they should be to preserve the visual Hamiltonian cycle, with $M_0$ naturally anchored at the top. 

When visualizing the $(9,4)$ system (and subsequently its mirror, the $(8,3)$ system), presenting both the abstract topology and the physical pitch-space is strictly necessary. Unlike the $(9,8)$ system, these exotic clusters share no structural familiarity with Western triads. Therefore, the dual layouts below serve specific, non-overlapping analytical functions:

\begin{enumerate}
    \item \textbf{Anchoring the Mathematical Root:} In the $(9,4)$ system, the $M_0$ chord is an exotic cluster $\{0, 9, 1\}$, but its structural root remains the pitch class $0$ (C). The Pitch-Space layout (right graph) strictly anchors these unfamiliar chords to their respective roots on the traditional Circle of Fifths, providing a fixed physical reference point on a standard 12-tone instrument.
    \item \textbf{Mapping Physical Effort (Geometric Rewiring):} The Topological layout (left graph) proves that smooth voice-leading is logically preserved (manifesting as a single adjacent step in the graph). However, the Pitch-Space layout visualizes the actual physical distance required to execute that step on a standard instrument. It reveals how a single logical progression structurally forces large, non-linear leaps across the physical keyboard.
    \item \textbf{The Physical Hardware Proof:} Projecting these abstract clusters onto the physical Circle of Fifths provides a concrete proof of symmetry between dual systems. As will be demonstrated, despite the severe rewiring of paths seen in the $(9,4)$ Pitch-Space graph, its inverse $(8,3)$ system exhibits the exact same physical geometry. This proves that performing in either universe requires an identical underlying physical architecture (or isomorphic instrument wiring), merely traversed with reversed voice-leading pathways.
\end{enumerate}
\vspace{0.4cm}

\begin{figure}[H]
    \centering
    \begin{minipage}{0.48\textwidth}
        \centering
        \begin{tikzpicture}[scale=0.8, transform shape]
            \tikzset{snode/.style={circle, draw=black, fill=white, inner sep=1.2pt}}
            \def\n{12} \def\R{2.0} \def\Rtext{2.3} 
            
            \foreach \s in {0,...,11} {
                \pgfmathsetmacro{\k}{int(mod(5*\s, 12))}
                \pgfmathsetmacro{\j}{int(mod(9 + 5*\s, 12))}
                \pgfmathsetmacro{\angM}{90 + (\s*2)*15 + 30}
                \pgfmathsetmacro{\angm}{90 + (\s*2 + 1)*15 + 30}
                \coordinate (M\k) at (\angM:\R);
                \coordinate (m\j) at (\angm:\R);
            }
            \foreach \i in {0,...,11} {
                \draw[red, thick] (M\i) -- (m\i);
                \pgfmathsetmacro{\idxL}{int(mod(\i+4, \n))}
                \draw[blue, thick] (M\i) -- (m\idxL);
                \pgfmathsetmacro{\idxR}{int(mod(\i+9, \n))}
                \draw[levigold, thick] (M\i) -- (m\idxR);
            }
            \foreach \s in {0,...,11} {
                \pgfmathsetmacro{\k}{int(mod(5*\s, 12))}
                \pgfmathsetmacro{\j}{int(mod(9 + 5*\s, 12))}
                \pgfmathsetmacro{\angM}{90 + (\s*2)*15 + 30}
                \pgfmathsetmacro{\angm}{90 + (\s*2 + 1)*15 + 30}
          
                \node[snode] at (M\k) {};
                \node[snode] at (m\j) {};
                
                \pgfmathsetmacro{\labelK}{int(mod(5*\k + 1, 12))}
                \pgfmathsetmacro{\labelJ}{int(mod(5*\j + 1, 12))}
                
                \pgfmathsetmacro{\normM}{mod(\angM+360, 360)}
                \pgfmathsetmacro{\normm}{mod(\angm+360, 360)}
                \pgfmathsetmacro{\rotM}{ (\normM > 90.1 && \normM < 269.9) ? \angM - 180 : \angM }
                \pgfmathsetmacro{\rotm}{ (\normm > 90.1 && \normm < 269.9) ? \angm - 180 : \angm }
                
                \pgfmathsetmacro{\leftM}{ (\normM > 90.1 && \normM < 269.9) ? 1 : 0 }
                \ifnum\leftM=1
                    \node[font=\fontsize{8}{9}\selectfont, anchor=east, rotate=\rotM] at (\angM:\Rtext) {$M_{\labelK}$};
                \else
                    \node[font=\fontsize{8}{9}\selectfont, anchor=west, rotate=\rotM] at (\angM:\Rtext) {$M_{\labelK}$};
                \fi

                \pgfmathsetmacro{\leftm}{ (\normm > 90.1 && \normm < 269.9) ? 1 : 0 }
                \ifnum\leftm=1
                    \node[font=\fontsize{8}{9}\selectfont, anchor=east, rotate=\rotm] at (\angm:\Rtext) {$m_{\labelJ}$};
                \else
                    \node[font=\fontsize{8}{9}\selectfont, anchor=west, rotate=\rotm] at (\angm:\Rtext) {$m_{\labelJ}$};
                \fi
            }
        \end{tikzpicture}
        \caption*{Topological Layout \\ (Proof of Isomorphism)}
    \end{minipage}\hfill
    \begin{minipage}{0.48\textwidth}
        \centering
        \begin{tikzpicture}[scale=0.8, transform shape]
            \tikzset{snode/.style={circle, draw=black, fill=white, inner sep=1.2pt}}
            \def\n{12} \def\R{2.0} \def\Rtext{2.3} 
          
            \foreach \s in {0,...,11} {
                \pgfmathsetmacro{\k}{int(mod(5*\s, 12))}
                \pgfmathsetmacro{\j}{int(mod(9 + 5*\s, 12))}
                \pgfmathsetmacro{\angM}{90 + (\s*2)*15}
                \pgfmathsetmacro{\angm}{90 + (\s*2 + 1)*15}
        
                \coordinate (M\k) at (\angM:\R);
                \coordinate (m\j) at (\angm:\R);
            }
            \foreach \i in {0,...,11} {
                \draw[red, thick] (M\i) -- (m\i);
                \pgfmathsetmacro{\idxL}{int(mod(\i+8, \n))} 
                \draw[blue, thick] (M\i) -- (m\idxL);
                \pgfmathsetmacro{\idxR}{int(mod(\i+9, \n))}
                \draw[levigold, thick] (M\i) -- (m\idxR);
            }
            \foreach \s in {0,...,11} {
                \pgfmathsetmacro{\k}{int(mod(5*\s, 12))}
                \pgfmathsetmacro{\j}{int(mod(9 + 5*\s, 12))}
                \pgfmathsetmacro{\angM}{90 + (\s*2)*15}
                \pgfmathsetmacro{\angm}{90 + (\s*2 + 1)*15}
          
                \node[snode] at (M\k) {};
                \node[snode] at (m\j) {};
                \pgfmathsetmacro{\normM}{mod(\angM+360, 360)}
                \pgfmathsetmacro{\normm}{mod(\angm+360, 360)}
                \pgfmathsetmacro{\rotM}{ (\normM > 90.1 && \normM < 269.9) ? \angM - 180 : \angM }
                \pgfmathsetmacro{\rotm}{ (\normm > 90.1 && \normm < 269.9) ? \angm - 180 : \angm }
                
                \pgfmathsetmacro{\leftM}{ (\normM > 90.1 && \normM < 269.9) ? 1 : 0 }
                \ifnum\leftM=1
                    \node[font=\fontsize{8}{9}\selectfont, anchor=east, rotate=\rotM] at (\angM:\Rtext) {$M_{\k}$};
                \else
                    \node[font=\fontsize{8}{9}\selectfont, anchor=west, rotate=\rotM] at (\angM:\Rtext) {$M_{\k}$};
                \fi

                \pgfmathsetmacro{\leftm}{ (\normm > 90.1 && \normm < 269.9) ? 1 : 0 }
                \ifnum\leftm=1
                    \node[font=\fontsize{8}{9}\selectfont, anchor=east, rotate=\rotm] at (\angm:\Rtext) {$m_{\j}$};
                \else
                    \node[font=\fontsize{8}{9}\selectfont, anchor=west, rotate=\rotm] at (\angm:\Rtext) {$m_{\j}$};
                \fi
            }
        \end{tikzpicture}
        \caption*{Pitch-Space Layout \\ (Geometric Rewiring)}
    \end{minipage}
    \caption{Two views of the (9,4) system. Left: Thanks to the orientation-preserving affine map $5x+1$, the topological layout naturally preserves the (4,3) shape with $M_0$ exactly at 12 o'clock. Right: Anchoring chords to the Circle of Fifths reveals the scrambled voice-leading paths.}
\end{figure}

\subsubsection{System (8,3): The Mirror Universe of the Fundamental Anchor}

Just as the classical (4,3) system has its perfect mathematical twin in (9,8), the fundamental anchor also has a dark mirror. In the (8,3) system, the generator is a major seventh ($q=11$). 

{\tiny \begin{table}[H]
    \centering
    \begin{minipage}{0.48\textwidth}
        \centering
        \begin{tabular}{l}
        \toprule
        \textbf{Numerical ($M_r$)} \\
        \midrule
        $M_0 = \{\text{C, A}^\flat, \text{B}\} = \{0, 8, 11\}$ \\
        $M_1 = \{\text{C}^\sharp, \text{A, C}\} = \{1, 9, 0\}$ \\
        $M_2 = \{\text{D, B}^\flat, \text{C}^\sharp\} = \{2, 10, 1\}$ \\
        $M_3 = \{\text{E}^\flat, \text{B, D}\} = \{3, 11, 2\}$ \\
        $M_4 = \{\text{E, C, E}^\flat\} = \{4, 0, 3\}$ \\
        $M_5 = \{\text{F, C}^\sharp, \text{E}\} = \{5, 1, 4\}$ \\
        $M_6 = \{\text{F}^\sharp, \text{D, F}\} = \{6, 2, 5\}$ \\
        $M_7 = \{\text{G, E}^\flat, \text{F}^\sharp\} = \{7, 3, 6\}$ \\
        $M_8 = \{\text{A}^\flat, \text{E, G}\} = \{8, 4, 7\}$ \\
        $M_9 = \{\text{A, F, A}^\flat\} = \{9, 5, 8\}$ \\
        $M_{10} = \{\text{B}^\flat, \text{F}^\sharp, \text{A}\} = \{10, 6, 9\}$ \\
        $M_{11} = \{\text{B, G, B}^\flat\} = \{11, 7, 10\}$ \\
        \bottomrule
        \end{tabular}
    \end{minipage}\hfill
    \begin{minipage}{0.48\textwidth}
        \centering
        \begin{tabular}{l}
        \toprule
        \textbf{Numerical ($m_r$)} \\
        \midrule
        $m_0 = \{\text{C, E}^\flat, \text{B}\} = \{0, 3, 11\}$ \\
        $m_1 = \{\text{C}^\sharp, \text{E, C}\} = \{1, 4, 0\}$ \\
        $m_2 = \{\text{D, F, C}^\sharp\} = \{2, 5, 1\}$ \\
        $m_3 = \{\text{E}^\flat, \text{F}^\sharp, \text{D}\} = \{3, 6, 2\}$ \\
        $m_4 = \{\text{E, G, E}^\flat\} = \{4, 7, 3\}$ \\
        $m_5 = \{\text{F, A}^\flat, \text{E}\} = \{5, 8, 4\}$ \\
        $m_6 = \{\text{F}^\sharp, \text{A, F}\} = \{6, 9, 5\}$ \\
        $m_7 = \{\text{G, B}^\flat, \text{F}^\sharp\} = \{7, 10, 6\}$ \\
        $m_8 = \{\text{A}^\flat, \text{B, G}\} = \{8, 11, 7\}$ \\
        $m_9 = \{\text{A, C, A}^\flat\} = \{9, 0, 8\}$ \\
        $m_{10} = \{\text{B}^\flat, \text{C}^\sharp, \text{A}\} = \{10, 1, 9\}$ \\
        $m_{11} = \{\text{B, D, B}^\flat\} = \{11, 2, 10\}$ \\
        \bottomrule
        \end{tabular}
    \end{minipage}
    \caption{Chords in the (8,3) system.}
\end{table}}

As directly implied by the affine note bijections demonstrated in Theorem 2, the exotic chords of the (8,3) system are identical as sets to the exotic chords of the (9,4) system. They are merely rearranged with a different base tone. For example, $M_1$ in (8,3) is $\{1, 9, 0\}$, which contains the exact same notes as $M_0$ in (9,4), which is $\{0, 9, 1\}$. Every single chord in (8,3) is the exact same set of pitch classes as the (9,4) chord rooted a semitone lower.

Now that we have identified the internal composition of these shifted chordal sets, we can trace their allowable harmonic pathways across the Tonnetz.

{\tiny \begin{table}[H]
    \centering
    \begin{minipage}{0.48\textwidth}
        \centering
        \begin{tabular}{l}
        \toprule
        \textbf{Major $\to$ Minor} \\
        \midrule
        $M_0 \to m_0, m_8, m_9$ \\
        $M_1 \to m_1, m_9, m_{10}$ \\
        $M_2 \to m_2, m_{10}, m_{11}$ \\
        $M_3 \to m_3, m_{11}, m_0$ \\
        $M_4 \to m_4, m_0, m_1$ \\
        $M_5 \to m_5, m_1, m_2$ \\
        $M_6 \to m_6, m_2, m_3$ \\
        $M_7 \to m_7, m_3, m_4$ \\
        $M_8 \to m_8, m_4, m_5$ \\
        $M_9 \to m_9, m_5, m_6$ \\
        $M_{10} \to m_{10}, m_6, m_7$ \\
        $M_{11} \to m_{11}, m_7, m_8$ \\
        \bottomrule
        \end{tabular}
    \end{minipage}\hfill
    \begin{minipage}{0.48\textwidth}
        \centering
        \begin{tabular}{l}
        \toprule
        \textbf{Minor $\to$ Major} \\
        \midrule
        $m_0 \to M_0, M_4, M_3$ \\
        $m_1 \to M_1, M_5, M_4$ \\
        $m_2 \to M_2, M_6, M_5$ \\
        $m_3 \to M_3, M_7, M_6$ \\
        $m_4 \to M_4, M_8, M_7$ \\
        $m_5 \to M_5, M_9, M_8$ \\
        $m_6 \to M_6, M_{10}, M_9$ \\
        $m_7 \to M_7, M_{11}, M_{10}$ \\
        $m_8 \to M_8, M_0, M_{11}$ \\
        $m_9 \to M_9, M_1, M_0$ \\
        $m_{10} \to M_{10}, M_2, M_1$ \\
        $m_{11} \to M_{11}, M_3, M_2$ \\
        \bottomrule
        \end{tabular}
    \end{minipage}
    \caption{Tonnetz progression map for the (8,3) system.}
\end{table}}

In perfect accordance with the orientation-reversing mapping connecting them (derived in Theorem 2), examining these progressions reveals the exact same swapping phenomenon observed in the classical pair. In the (9,4) system, $M_r$ connects to $m_r, m_{r+9}$, and $m_{r+8}$. In this (8,3) system, $M_r$ connects to exactly the same set of chords: $m_r, m_{r+8}$, and $m_{r+9}$. The destinations are perfectly identical, but the left and right voice-leading paths have been exchanged.

\begin{figure}[H]
    \centering
    \begin{minipage}{0.48\textwidth}
        \centering
        \begin{tikzpicture}[scale=0.7, transform shape]
            \tikzset{snode/.style={circle, draw=black, fill=white, inner sep=1.2pt}}
            \def\n{12} \def\R{2.0} \def\Rtext{2.3} 
            
            \foreach \s in {0,...,11} {
                \pgfmathsetmacro{\k}{int(mod(5*\s, 12))}
                \pgfmathsetmacro{\j}{int(mod(9 + 5*\s, 12))}
                \pgfmathsetmacro{\angM}{90 + (\s*2)*15}
                \pgfmathsetmacro{\angm}{90 + (\s*2 + 1)*15}
                \coordinate (M\k) at (\angM:\R);
                \coordinate (m\j) at (\angm:\R);
            }
            \foreach \i in {0,...,11} {
                \draw[red, thick] (M\i) -- (m\i);
                \pgfmathsetmacro{\idxL}{int(mod(\i+4, \n))}
                \draw[blue, thick] (M\i) -- (m\idxL);
                \pgfmathsetmacro{\idxR}{int(mod(\i+9, \n))}
                \draw[levigold, thick] (M\i) -- (m\idxR);
            }
            \foreach \s in {0,...,11} {
                \pgfmathsetmacro{\k}{int(mod(5*\s, 12))}
                \pgfmathsetmacro{\j}{int(mod(9 + 5*\s, 12))}
                \pgfmathsetmacro{\angM}{90 + (\s*2)*15}
                \pgfmathsetmacro{\angm}{90 + (\s*2 + 1)*15}
          
                \node[snode] at (M\k) {};
                \node[snode] at (m\j) {};
                
                \pgfmathsetmacro{\labelK}{int(mod(5*\k, 12))}
                \pgfmathsetmacro{\labelJ}{int(mod(5*\j, 12))}
                
                \pgfmathsetmacro{\normM}{mod(\angM+360, 360)}
                \pgfmathsetmacro{\normm}{mod(\angm+360, 360)}
                \pgfmathsetmacro{\rotM}{ (\normM > 90.1 && \normM < 269.9) ? \angM - 180 : \angM }
                \pgfmathsetmacro{\rotm}{ (\normm > 90.1 && \normm < 269.9) ? \angm - 180 : \angm }
                
                \pgfmathsetmacro{\leftM}{ (\normM > 90.1 && \normM < 269.9) ? 1 : 0 }
                \ifnum\leftM=1
                    \node[font=\fontsize{8}{9}\selectfont, anchor=east, rotate=\rotM] at (\angM:\Rtext) {$M_{\labelK}$};
                \else
                    \node[font=\fontsize{8}{9}\selectfont, anchor=west, rotate=\rotM] at (\angM:\Rtext) {$M_{\labelK}$};
                \fi

                \pgfmathsetmacro{\leftm}{ (\normm > 90.1 && \normm < 269.9) ? 1 : 0 }
                \ifnum\leftm=1
                    \node[font=\fontsize{8}{9}\selectfont, anchor=east, rotate=\rotm] at (\angm:\Rtext) {$m_{\labelJ}$};
                \else
                    \node[font=\fontsize{8}{9}\selectfont, anchor=west, rotate=\rotm] at (\angm:\Rtext) {$m_{\labelJ}$};
                \fi
            }
        \end{tikzpicture}
        \caption*{Topological Layout \\ (Proof of Isomorphism)}
    \end{minipage}\hfill
    \begin{minipage}{0.48\textwidth}
        \centering
        \begin{tikzpicture}[scale=0.7, transform shape]
            \tikzset{snode/.style={circle, draw=black, fill=white, inner sep=1.2pt}}
            \def\n{12} \def\R{2.0} \def\Rtext{2.3} 
          
            \foreach \s in {0,...,11} {
                \pgfmathsetmacro{\k}{int(mod(5*\s, 12))}
                \pgfmathsetmacro{\j}{int(mod(9 + 5*\s, 12))}
                \pgfmathsetmacro{\angM}{90 + (\s*2)*15}
                \pgfmathsetmacro{\angm}{90 + (\s*2 + 1)*15}
        
                \coordinate (M\k) at (\angM:\R);
                \coordinate (m\j) at (\angm:\R);
            }
            \foreach \i in {0,...,11} {
                \draw[red, thick] (M\i) -- (m\i);
                \pgfmathsetmacro{\idxL}{int(mod(\i+8, \n))} 
                \draw[blue, thick] (M\i) -- (m\idxL);
                \pgfmathsetmacro{\idxR}{int(mod(\i+9, \n))}
                \draw[levigold, thick] (M\i) -- (m\idxR);
            }
            \foreach \s in {0,...,11} {
                \pgfmathsetmacro{\k}{int(mod(5*\s, 12))}
                \pgfmathsetmacro{\j}{int(mod(9 + 5*\s, 12))}
                \pgfmathsetmacro{\angM}{90 + (\s*2)*15}
                \pgfmathsetmacro{\angm}{90 + (\s*2 + 1)*15}
          
                \node[snode] at (M\k) {};
                \node[snode] at (m\j) {};
                \pgfmathsetmacro{\normM}{mod(\angM+360, 360)}
                \pgfmathsetmacro{\normm}{mod(\angm+360, 360)}
                \pgfmathsetmacro{\rotM}{ (\normM > 90.1 && \normM < 269.9) ? \angM - 180 : \angM }
                \pgfmathsetmacro{\rotm}{ (\normm > 90.1 && \normm < 269.9) ? \angm - 180 : \angm }
                
                \pgfmathsetmacro{\leftM}{ (\normM > 90.1 && \normM < 269.9) ? 1 : 0 }
                \ifnum\leftM=1
                    \node[font=\fontsize{8}{9}\selectfont, anchor=east, rotate=\rotM] at (\angM:\Rtext) {$M_{\k}$};
                \else
                    \node[font=\fontsize{8}{9}\selectfont, anchor=west, rotate=\rotM] at (\angM:\Rtext) {$M_{\k}$};
                \fi

                \pgfmathsetmacro{\leftm}{ (\normm > 90.1 && \normm < 269.9) ? 1 : 0 }
                \ifnum\leftm=1
                    \node[font=\fontsize{8}{9}\selectfont, anchor=east, rotate=\rotm] at (\angm:\Rtext) {$m_{\j}$};
                \else
                    \node[font=\fontsize{8}{9}\selectfont, anchor=west, rotate=\rotm] at (\angm:\Rtext) {$m_{\j}$};
                \fi
            }
        \end{tikzpicture}
        \caption*{Pitch-Space Layout \\ (Geometric Rewiring)}
    \end{minipage}
    \caption{Two views of the (8,3) system. As shown by the Pitch-Space graph on the right, this system is geometrically identical to the (9,4) system, only traversed with swapped voice-leading paths.}
    \label{fig:fundamental_mirror}
\end{figure}

The visual relationship between the $(9,4)$ and $(8,3)$ graphs perfectly echoes the dual relationship between the classical $(4,3)$ and its negative $(9,8)$. As seen in Figure \ref{fig:fundamental_mirror}, the physical Pitch-Space geometry of $(8,3)$ is completely identical to that of the $(9,4)$ system, with the sole exception that the $L$ (blue) and $R$ (gold) voice-leading paths have swapped places. 

As we proved in Theorem 2, the transformation connecting these two exotic systems requires the relative multiplier $5 \times 7 \equiv -1 \pmod{12}$. Thus, the fundamental $(9,4)$ anchor also possesses its own negative mirror dimension, making the $(8,3)$ harmonic system the exact pitch-inversion of the $(9,4)$ system.

Because the $(8,3)$ system is merely the negative reflection of the $(9,4)$ system, it utilises the exact same 24 triads, reinterpreted from a different root. The $(8,3)$ chords are identical as sets of pitch classes to the $(9,4)$ chords shifted downwards by a minor second:
\begin{equation}
    M^{(8,3)}_k = M^{(9,4)}_{k-1} \quad \text{and} \quad m^{(8,3)}_k = m^{(9,4)}_{k-1}.
\end{equation}
Thus, navigating the $(8,3)$ system is physically equivalent to playing the exotic chords of the $(9,4)$ universe, but shifted a semitone lower and traversed in reverse.

\begin{conclusion}[The Musician's Playbook: Navigating the Alien Clusters]
Unlike the familiar inversions of the $(9,8)$ system, improvising within the $(9,4)$ and $(8,3)$ systems requires the musician to physically master entirely new harmonic objects—dense, alien clusters driven by semitone relationships. At first glance, this seems daunting. However, the geometric duality provides a massive shortcut for the performer.

You only need to master the physical hand shapes (the ``grips") and Tonnetz pathways for the foundational $(9,4)$ system. Once your muscle memory adapts to these exotic clusters, \textbf{you gain the $(8,3)$ universe entirely for free}. To transpose a $(9,4)$ progression into the $(8,3)$ mirror universe, you simply shift your entire hand physically down by one semitone (a minor second) and, just as with Negative Harmony, swap your Left/Right (Blue/Gold) voice-leading reflexes. The $(8,3)$ system is not a new set of chords to memorize; it is simply the $(9,4)$ system played one key lower and driven in reverse.
\end{conclusion}

\subsubsection{Acoustic Realities: Auditory Proofs of the Isomorphic Universes}

To verify that these topological symmetries constitute genuinely playable musical environments, we translated our algebraic framework into a custom computational engine. Rather than relying solely on theoretical network graphs, this script processes standard MIDI performances of classical repertoire, applying the appropriate affine bijection $f(x) \equiv ax + b \pmod{12}$ directly to every individual pitch class. Crucially, the algorithm perfectly isolates the harmonic transformation: it seamlessly preserves the original temporal rhythms, dynamic phrasing, and octave registers of the compositions. To ensure the highest acoustic fidelity and provide a resonant, natural canvas for these unfamiliar chromatic spaces, the transformed MIDI sequences were rendered into audio utilizing the advanced physical modeling synthesizer, Pianoteq 9 PRO.

The resulting recordings offer a profound sensory verification of our mathematical classification. While the underlying voice-leading logic remains structurally flawless, the acoustic surface is radically rewired. To cleanly demonstrate the unique character of each isomorphic dimension without redundancy, we mapped three distinct classical works into our exotic universes. 

\vspace{0.2cm}
\noindent\textbf{Audio Format \& Listening Guide:}\\
Each ``.wav'' audio document linked below presents a continuous, sequential rendering of a single piece. The track begins with the original $(4,3)$ classical system, immediately followed by its exact mathematical transformations mapped into the $(9,8)$, $(9,4)$, and $(8,3)$ systems, respectively, separated by brief pauses. We encourage the listener to pay close attention to the acoustic `duality' present in the recordings: first, the mirrored relationship between the familiar $(4,3)$ chords and their inverted $(9,8)$ counterparts (Negative Harmony), and second, the structural and geometric equivalence between the dense, alien clusters of the $(9,4)$ and $(8,3)$ universes.

\begin{itemize}
    \item \textbf{George Frideric Händel -- Passacaglia in G minor (HWV 432)}\\
    \href{https://www.fuw.edu.pl/~nurowski/Haendel_Passacaglia_in_G_minor_all_systems.wav}{Listen to all four system versions (.wav)}
    
    \vspace{0.2cm}
    \item \textbf{Franz Schubert -- Piano Trio No. 2 in E-flat major, Op. 100 (2nd Movement)}\\
    \href{https://www.fuw.edu.pl/~nurowski/Schubert_Piano_Trio_2mvt_all_systems.wav}{Listen to all four system versions (.wav)}
    
    \vspace{0.2cm}
    \item \textbf{Arnold Schoenberg -- Five Piano Pieces, Op. 23}\\
    \href{https://www.fuw.edu.pl/~nurowski/Schoenberg_Piano_Pieces_all_systems.wav}{Listen to all four system versions (.wav)}
\end{itemize}

These auditory examples confirm that the structural equivalence predicted theoretically by the Chinese Remainder Theorem and Levi graphs is not merely an abstract mathematical curiosity. It manifests as a tangible, rigidly coherent, and deeply fascinating auditory universe.

\subsection{The 10-TET Decaphonic Universe}

Transitioning from the familiar acoustic territory of twelve tones, we now enter the more abstract decaphonic universe of 10-TET. As established in Section \ref{sec:crt}, this space is fundamentally structured by the composite nature of $10 = 2 \times 5$. Lacking traditional Western note names, we will strictly map out its chords and pathways using the numerical pitch classes $0, 1, \dots, 9$. 

Our algebraic analysis yielded exactly four non-degenerate harmonic systems in this universe: $(6,5)$, $(8,5)$, $(2,5)$, and $(4,5)$. Let us navigate them sequentially, maintaining the same visual dual-layout framework we established for 12-TET to observe how the structural geometry morphs under different affine transformations.

\subsubsection{The Fundamental Anchor: System (6,5)}

We begin with the absolute algebraic baseline of 10-TET: the $(6,5)$ system. Driven by the simple decaphonic step ($q = 1$) acting as the generator, this universe serves as the foundational anchor for all other decaphonic topologies. The explicit pitch-class contents of its twenty exotic triads are defined as follows:

{\tiny \begin{table}[H]
    \centering
    \begin{minipage}{0.48\textwidth}
        \centering
        \begin{tabular}{l}
        \toprule
        \textbf{Numerical ($M_r$)} \\
        \midrule
        $M_0 = \{0, 6, 1\}$ \\
        $M_1 = \{1, 7, 2\}$ \\
        $M_2 = \{2, 8, 3\}$ \\
        $M_3 = \{3, 9, 4\}$ \\
        $M_4 = \{4, 0, 5\}$ \\
        $M_5 = \{5, 1, 6\}$ \\
        $M_6 = \{6, 2, 7\}$ \\
        $M_7 = \{7, 3, 8\}$ \\
        $M_8 = \{8, 4, 9\}$ \\
        $M_9 = \{9, 5, 0\}$ \\
        \bottomrule
        \end{tabular}
    \end{minipage}\hfill
    \begin{minipage}{0.48\textwidth}
        \centering
        \begin{tabular}{l}
        \toprule
        \textbf{Numerical ($m_r$)} \\
        \midrule
        $m_0 = \{0, 5, 1\}$ \\
        $m_1 = \{1, 6, 2\}$ \\
        $m_2 = \{2, 7, 3\}$ \\
        $m_3 = \{3, 8, 4\}$ \\
        $m_4 = \{4, 9, 5\}$ \\
        $m_5 = \{5, 0, 6\}$ \\
        $m_6 = \{6, 1, 7\}$ \\
        $m_7 = \{7, 2, 8\}$ \\
        $m_8 = \{8, 3, 9\}$ \\
        $m_9 = \{9, 4, 0\}$ \\
        \bottomrule
        \end{tabular}
    \end{minipage}
    \caption{Chords in the foundational (6,5) system.}
\end{table}}

With the internal composition of the triads established, we can map their voice-leading pathways across the Tonnetz. Just as in the classical Western system, a Major chord $M_r$ flows gracefully into three specific minor chords. For the $(6,5)$ system, these immediate destinations are $m_r, m_{r+6}$, and $m_{r+5}$.

{\tiny \begin{table}[H]
    \centering
    \begin{minipage}{0.48\textwidth}
        \centering
        \begin{tabular}{l}
        \toprule
        \textbf{Major $\to$ Minor} \\
        \midrule
        $M_0 \to m_0, m_6, m_5$ \\
        $M_1 \to m_1, m_7, m_6$ \\
        $M_2 \to m_2, m_8, m_7$ \\
        $M_3 \to m_3, m_9, m_8$ \\
        $M_4 \to m_4, m_0, m_9$ \\
        $M_5 \to m_5, m_1, m_0$ \\
        $M_6 \to m_6, m_2, m_1$ \\
        $M_7 \to m_7, m_3, m_2$ \\
        $M_8 \to m_8, m_4, m_3$ \\
        $M_9 \to m_9, m_5, m_4$ \\
        \bottomrule
        \end{tabular}
    \end{minipage}\hfill
    \begin{minipage}{0.48\textwidth}
        \centering
        \begin{tabular}{l}
        \toprule
        \textbf{Minor $\to$ Major} \\
        \midrule
        $m_0 \to M_0, M_4, M_5$ \\
        $m_1 \to M_1, M_5, M_6$ \\
        $m_2 \to M_2, M_6, M_7$ \\
        $m_3 \to M_3, M_7, M_8$ \\
        $m_4 \to M_4, M_8, M_9$ \\
        $m_5 \to M_5, M_9, M_0$ \\
        $m_6 \to M_6, M_0, M_1$ \\
        $m_7 \to M_7, M_1, M_2$ \\
        $m_8 \to M_8, M_2, M_3$ \\
        $m_9 \to M_9, M_3, M_4$ \\
        \bottomrule
        \end{tabular}
    \end{minipage}
    \caption{Tonnetz progression map for the (6,5) system.}
\end{table}}

\begin{figure}[H]
    \centering
    \begin{minipage}{0.48\textwidth}
        \centering
        \begin{tikzpicture}[scale=0.8, transform shape]
            \tikzset{snode/.style={circle, draw=black, fill=white, inner sep=1.2pt}}
            \def\R{2.0} \def\Rtext{2.3} 
            
            \foreach \s in {0,...,9} {
                \pgfmathsetmacro{\k}{int(\s)}
                \pgfmathsetmacro{\j}{int(mod(\s + 6, 10))}
                \pgfmathsetmacro{\angM}{90 - (\s*2)*18}
                \pgfmathsetmacro{\angm}{90 - (\s*2 + 1)*18}
                \coordinate (M\k) at (\angM:\R);
                \coordinate (m\j) at (\angm:\R);
            }
            \foreach \i in {0,...,9} {
                \draw[red, thick] (M\i) -- (m\i);
                \pgfmathsetmacro{\idxL}{int(mod(\i+6, 10))}
                \draw[blue, thick] (M\i) -- (m\idxL);
                \pgfmathsetmacro{\idxR}{int(mod(\i+5, 10))}
                \draw[levigold, thick] (M\i) -- (m\idxR);
            }
            \foreach \s in {0,...,9} {
                \pgfmathsetmacro{\k}{int(\s)}
                \pgfmathsetmacro{\j}{int(mod(\s + 6, 10))}
                \pgfmathsetmacro{\angM}{90 - (\s*2)*18}
                \pgfmathsetmacro{\angm}{90 - (\s*2 + 1)*18}
          
                \node[snode] at (M\k) {};
                \node[snode] at (m\j) {};
                
                \pgfmathsetmacro{\normM}{mod(\angM+360, 360)}
                \pgfmathsetmacro{\normm}{mod(\angm+360, 360)}
                \pgfmathsetmacro{\rotM}{ (\normM > 90.1 && \normM < 269.9) ? \angM - 180 : \angM }
                \pgfmathsetmacro{\rotm}{ (\normm > 90.1 && \normm < 269.9) ? \angm - 180 : \angm }
                
                \pgfmathsetmacro{\leftM}{ (\normM > 90.1 && \normM < 269.9) ? 1 : 0 }
                \ifnum\leftM=1
                    \node[font=\fontsize{8}{9}\selectfont, anchor=east, rotate=\rotM] at (\angM:\Rtext) {$M_{\k}$};
                \else
                    \node[font=\fontsize{8}{9}\selectfont, anchor=west, rotate=\rotM] at (\angM:\Rtext) {$M_{\k}$};
                \fi

                \pgfmathsetmacro{\leftm}{ (\normm > 90.1 && \normm < 269.9) ? 1 : 0 }
                \ifnum\leftm=1
                    \node[font=\fontsize{8}{9}\selectfont, anchor=east, rotate=\rotm] at (\angm:\Rtext) {$m_{\j}$};
                \else
                    \node[font=\fontsize{8}{9}\selectfont, anchor=west, rotate=\rotm] at (\angm:\Rtext) {$m_{\j}$};
                \fi
            }
        \end{tikzpicture}
        \caption*{Topological Layout \\ (Proof of Isomorphism)}
    \end{minipage}\hfill
    \begin{minipage}{0.48\textwidth}
        \centering
        \begin{tikzpicture}[scale=0.8, transform shape]
            \tikzset{snode/.style={circle, draw=black, fill=white, inner sep=1.2pt}}
            \def\R{2.0} \def\Rtext{2.3} 
            
            \foreach \s in {0,...,9} {
                \pgfmathsetmacro{\k}{int(\s)}
                \pgfmathsetmacro{\j}{int(mod(\s + 6, 10))}
                \pgfmathsetmacro{\angM}{90 - (\s*2)*18}
                \pgfmathsetmacro{\angm}{90 - (\s*2 + 1)*18}
                \coordinate (M\k) at (\angM:\R);
                \coordinate (m\j) at (\angm:\R);
            }
            \foreach \i in {0,...,9} {
                \draw[red, thick] (M\i) -- (m\i);
                \pgfmathsetmacro{\idxL}{int(mod(\i+6, 10))}
                \draw[blue, thick] (M\i) -- (m\idxL);
                \pgfmathsetmacro{\idxR}{int(mod(\i+5, 10))}
                \draw[levigold, thick] (M\i) -- (m\idxR);
            }
            \foreach \s in {0,...,9} {
                \pgfmathsetmacro{\k}{int(\s)}
                \pgfmathsetmacro{\j}{int(mod(\s + 6, 10))}
                \pgfmathsetmacro{\angM}{90 - (\s*2)*18}
                \pgfmathsetmacro{\angm}{90 - (\s*2 + 1)*18}
          
                \node[snode] at (M\k) {};
                \node[snode] at (m\j) {};
                
                \pgfmathsetmacro{\normM}{mod(\angM+360, 360)}
                \pgfmathsetmacro{\normm}{mod(\angm+360, 360)}
                \pgfmathsetmacro{\rotM}{ (\normM > 90.1 && \normM < 269.9) ? \angM - 180 : \angM }
                \pgfmathsetmacro{\rotm}{ (\normm > 90.1 && \normm < 269.9) ? \angm - 180 : \angm }
                
                \pgfmathsetmacro{\leftM}{ (\normM > 90.1 && \normM < 269.9) ? 1 : 0 }
                \ifnum\leftM=1
                    \node[font=\fontsize{8}{9}\selectfont, anchor=east, rotate=\rotM] at (\angM:\Rtext) {$M_{\k}$};
                \else
                    \node[font=\fontsize{8}{9}\selectfont, anchor=west, rotate=\rotM] at (\angM:\Rtext) {$M_{\k}$};
                \fi

                \pgfmathsetmacro{\leftm}{ (\normm > 90.1 && \normm < 269.9) ? 1 : 0 }
                \ifnum\leftm=1
                    \node[font=\fontsize{8}{9}\selectfont, anchor=east, rotate=\rotm] at (\angm:\Rtext) {$m_{\j}$};
                \else
                    \node[font=\fontsize{8}{9}\selectfont, anchor=west, rotate=\rotm] at (\angm:\Rtext) {$m_{\j}$};
                \fi
            }
        \end{tikzpicture}
        \caption*{Pitch-Space Layout \\ (Geometric Rewiring)}
    \end{minipage}
    \caption{The foundational (6,5) system. Because this is the absolute anchor of the 10-TET universe, its topological progression perfectly matches physical chromatic distances. Thus, the Topological and Pitch-Space layouts are geometrically identical.}
\end{figure}

\subsubsection{System (8,5): The First Expansion}

Our first harmonic expansion occurs when we apply the linear multiplier $f(x) \equiv 3x \pmod{10}$. This generates the $(8,5)$ system, expanding the structural generator to $q=3$. While the underlying mathematical logic remains pristine, the physical intervals between the notes begin to stretch.

{\tiny \begin{table}[H]
    \centering
    \begin{minipage}{0.48\textwidth}
        \centering
        \begin{tabular}{l}
        \toprule
        \textbf{Numerical ($M_r$)} \\
        \midrule
        $M_0 = \{0, 8, 3\}$ \\
        $M_1 = \{1, 9, 4\}$ \\
        $M_2 = \{2, 0, 5\}$ \\
        $M_3 = \{3, 1, 6\}$ \\
        $M_4 = \{4, 2, 7\}$ \\
        $M_5 = \{5, 3, 8\}$ \\
        $M_6 = \{6, 4, 9\}$ \\
        $M_7 = \{7, 5, 0\}$ \\
        $M_8 = \{8, 6, 1\}$ \\
        $M_9 = \{9, 7, 2\}$ \\
        \bottomrule
        \end{tabular}
    \end{minipage}\hfill
    \begin{minipage}{0.48\textwidth}
        \centering
        \begin{tabular}{l}
        \toprule
        \textbf{Numerical ($m_r$)} \\
        \midrule
        $m_0 = \{0, 5, 3\}$ \\
        $m_1 = \{1, 6, 4\}$ \\
        $m_2 = \{2, 7, 5\}$ \\
        $m_3 = \{3, 8, 6\}$ \\
        $m_4 = \{4, 9, 7\}$ \\
        $m_5 = \{5, 0, 8\}$ \\
        $m_6 = \{6, 1, 9\}$ \\
        $m_7 = \{7, 2, 0\}$ \\
        $m_8 = \{8, 3, 1\}$ \\
        $m_9 = \{9, 4, 2\}$ \\
        \bottomrule
        \end{tabular}
    \end{minipage}
    \caption{Chords in the (8,5) system.}
\end{table}}

This systemic expansion naturally alters the progression pathways. We map the new permissible connections below:

{\tiny \begin{table}[H]
    \centering
    \begin{minipage}{0.48\textwidth}
        \centering
        \begin{tabular}{l}
        \toprule
        \textbf{Major $\to$ Minor} \\
        \midrule
        $M_0 \to m_0, m_8, m_5$ \\
        $M_1 \to m_1, m_9, m_6$ \\
        $M_2 \to m_2, m_0, m_7$ \\
        $M_3 \to m_3, m_1, m_8$ \\
        $M_4 \to m_4, m_2, m_9$ \\
        $M_5 \to m_5, m_3, m_0$ \\
        $M_6 \to m_6, m_4, m_1$ \\
        $M_7 \to m_7, m_5, m_2$ \\
        $M_8 \to m_8, m_6, m_3$ \\
        $M_9 \to m_9, m_7, m_4$ \\
        \bottomrule
        \end{tabular}
    \end{minipage}\hfill
    \begin{minipage}{0.48\textwidth}
        \centering
        \begin{tabular}{l}
        \toprule
        \textbf{Minor $\to$ Major} \\
        \midrule
        $m_0 \to M_0, M_2, M_5$ \\
        $m_1 \to M_1, M_3, M_6$ \\
        $m_2 \to M_2, M_4, M_7$ \\
        $m_3 \to M_3, M_5, M_8$ \\
        $m_4 \to M_4, M_6, M_9$ \\
        $m_5 \to M_5, M_7, M_0$ \\
        $m_6 \to M_6, M_8, M_1$ \\
        $m_7 \to M_7, M_9, M_2$ \\
        $m_8 \to M_8, M_0, M_3$ \\
        $m_9 \to M_9, M_1, M_4$ \\
        \bottomrule
        \end{tabular}
    \end{minipage}
    \caption{Tonnetz progression map for the (8,5) system.}
\end{table}}

The consequences of this mathematical expansion become visually apparent when plotted on our dual graphs. While the abstract topology (left) remains perfectly preserved under the isomorphism, the physical pitch-space mapping (right) reveals that executing these smooth logical progressions requires much wider, non-linear traversals across the instrument.

\begin{figure}[H]
    \centering
    \begin{minipage}{0.48\textwidth}
        \centering
        \begin{tikzpicture}[scale=0.8, transform shape]
            \tikzset{snode/.style={circle, draw=black, fill=white, inner sep=1.2pt}}
            \def\R{2.0} \def\Rtext{2.3} 
            
            \foreach \s in {0,...,9} {
                \pgfmathsetmacro{\k}{int(\s)}
                \pgfmathsetmacro{\j}{int(mod(\s + 6, 10))}
                \pgfmathsetmacro{\angM}{90 - (\s*2)*18}
                \pgfmathsetmacro{\angm}{90 - (\s*2 + 1)*18}
                \coordinate (M\k) at (\angM:\R);
                \coordinate (m\j) at (\angm:\R);
            }
            \foreach \i in {0,...,9} {
                \draw[red, thick] (M\i) -- (m\i);
                \pgfmathsetmacro{\idxL}{int(mod(\i+6, 10))}
                \draw[blue, thick] (M\i) -- (m\idxL);
                \pgfmathsetmacro{\idxR}{int(mod(\i+5, 10))}
                \draw[levigold, thick] (M\i) -- (m\idxR);
            }
            \foreach \s in {0,...,9} {
                \pgfmathsetmacro{\k}{int(\s)}
                \pgfmathsetmacro{\j}{int(mod(\s + 6, 10))}
                \pgfmathsetmacro{\angM}{90 - (\s*2)*18}
                \pgfmathsetmacro{\angm}{90 - (\s*2 + 1)*18}
          
                \node[snode] at (M\k) {};
                \node[snode] at (m\j) {};
                
                \pgfmathsetmacro{\labelK}{int(mod(3*\k, 10))}
                \pgfmathsetmacro{\labelJ}{int(mod(3*\j, 10))}
                
                \pgfmathsetmacro{\normM}{mod(\angM+360, 360)}
                \pgfmathsetmacro{\normm}{mod(\angm+360, 360)}
                \pgfmathsetmacro{\rotM}{ (\normM > 90.1 && \normM < 269.9) ? \angM - 180 : \angM }
                \pgfmathsetmacro{\rotm}{ (\normm > 90.1 && \normm < 269.9) ? \angm - 180 : \angm }
                
                \pgfmathsetmacro{\leftM}{ (\normM > 90.1 && \normM < 269.9) ? 1 : 0 }
                \ifnum\leftM=1
                    \node[font=\fontsize{8}{9}\selectfont, anchor=east, rotate=\rotM] at (\angM:\Rtext) {$M_{\labelK}$};
                \else
                    \node[font=\fontsize{8}{9}\selectfont, anchor=west, rotate=\rotM] at (\angM:\Rtext) {$M_{\labelK}$};
                \fi

                \pgfmathsetmacro{\leftm}{ (\normm > 90.1 && \normm < 269.9) ? 1 : 0 }
                \ifnum\leftm=1
                    \node[font=\fontsize{8}{9}\selectfont, anchor=east, rotate=\rotm] at (\angm:\Rtext) {$m_{\labelJ}$};
                \else
                    \node[font=\fontsize{8}{9}\selectfont, anchor=west, rotate=\rotm] at (\angm:\Rtext) {$m_{\labelJ}$};
                \fi
            }
        \end{tikzpicture}
        \caption*{Topological Layout \\ (Proof of Isomorphism)}
    \end{minipage}\hfill
    \begin{minipage}{0.48\textwidth}
        \centering
        \begin{tikzpicture}[scale=0.8, transform shape]
            \tikzset{snode/.style={circle, draw=black, fill=white, inner sep=1.2pt}}
            \def\R{2.0} \def\Rtext{2.3} 
            
            \foreach \s in {0,...,9} {
                \pgfmathsetmacro{\k}{int(\s)}
                \pgfmathsetmacro{\j}{int(mod(\s + 6, 10))}
                \pgfmathsetmacro{\angM}{90 - (\s*2)*18}
                \pgfmathsetmacro{\angm}{90 - (\s*2 + 1)*18}
                \coordinate (M\k) at (\angM:\R);
                \coordinate (m\j) at (\angm:\R);
            }
            \foreach \i in {0,...,9} {
                \draw[red, thick] (M\i) -- (m\i);
                \pgfmathsetmacro{\idxL}{int(mod(\i+8, 10))}
                \draw[blue, thick] (M\i) -- (m\idxL);
                \pgfmathsetmacro{\idxR}{int(mod(\i+5, 10))}
                \draw[levigold, thick] (M\i) -- (m\idxR);
            }
            \foreach \s in {0,...,9} {
                \pgfmathsetmacro{\k}{int(\s)}
                \pgfmathsetmacro{\j}{int(mod(\s + 6, 10))}
                \pgfmathsetmacro{\angM}{90 - (\s*2)*18}
                \pgfmathsetmacro{\angm}{90 - (\s*2 + 1)*18}
          
                \node[snode] at (M\k) {};
                \node[snode] at (m\j) {};
                
                \pgfmathsetmacro{\normM}{mod(\angM+360, 360)}
                \pgfmathsetmacro{\normm}{mod(\angm+360, 360)}
                \pgfmathsetmacro{\rotM}{ (\normM > 90.1 && \normM < 269.9) ? \angM - 180 : \angM }
                \pgfmathsetmacro{\rotm}{ (\normm > 90.1 && \normm < 269.9) ? \angm - 180 : \angm }
                
                \pgfmathsetmacro{\leftM}{ (\normM > 90.1 && \normM < 269.9) ? 1 : 0 }
                \ifnum\leftM=1
                    \node[font=\fontsize{8}{9}\selectfont, anchor=east, rotate=\rotM] at (\angM:\Rtext) {$M_{\k}$};
                \else
                    \node[font=\fontsize{8}{9}\selectfont, anchor=west, rotate=\rotM] at (\angM:\Rtext) {$M_{\k}$};
                \fi

                \pgfmathsetmacro{\leftm}{ (\normm > 90.1 && \normm < 269.9) ? 1 : 0 }
                \ifnum\leftm=1
                    \node[font=\fontsize{8}{9}\selectfont, anchor=east, rotate=\rotm] at (\angm:\Rtext) {$m_{\j}$};
                \else
                    \node[font=\fontsize{8}{9}\selectfont, anchor=west, rotate=\rotm] at (\angm:\Rtext) {$m_{\j}$};
                \fi
            }
        \end{tikzpicture}
        \caption*{Pitch-Space Layout \\ (Geometric Rewiring)}
    \end{minipage}
    \caption{Two views of the (8,5) system. The topological isomorphism (left) is guaranteed by $f(x) \equiv 3x \pmod{10}$, whilst the pitch-space mapping reveals the expanded physical traversal required by $t=8$ (blue paths).}
\end{figure}

\subsubsection{System (2,5): The Second Expansion}

Applying the multiplier $f(x) \equiv 7x \pmod{10}$ yields the $(2,5)$ system, with a generator of $q=7$. Functioning structurally as the counterpart to the $(8,5)$ universe, this system further scrambles the internal intervals of the triads. 

{\tiny \begin{table}[H]
    \centering
    \begin{minipage}{0.48\textwidth}
        \centering
        \begin{tabular}{l}
        \toprule
        \textbf{Numerical ($M_r$)} \\
        \midrule
        $M_0 = \{0, 2, 7\}$ \\
        $M_1 = \{1, 3, 8\}$ \\
        $M_2 = \{2, 4, 9\}$ \\
        $M_3 = \{3, 5, 0\}$ \\
        $M_4 = \{4, 6, 1\}$ \\
        $M_5 = \{5, 7, 2\}$ \\
        $M_6 = \{6, 8, 3\}$ \\
        $M_7 = \{7, 9, 4\}$ \\
        $M_8 = \{8, 0, 5\}$ \\
        $M_9 = \{9, 1, 6\}$ \\
        \bottomrule
        \end{tabular}
    \end{minipage}\hfill
    \begin{minipage}{0.48\textwidth}
        \centering
        \begin{tabular}{l}
        \toprule
        \textbf{Numerical ($m_r$)} \\
        \midrule
        $m_0 = \{0, 5, 7\}$ \\
        $m_1 = \{1, 6, 8\}$ \\
        $m_2 = \{2, 7, 9\}$ \\
        $m_3 = \{3, 8, 0\}$ \\
        $m_4 = \{4, 9, 1\}$ \\
        $m_5 = \{5, 0, 2\}$ \\
        $m_6 = \{6, 1, 3\}$ \\
        $m_7 = \{7, 2, 4\}$ \\
        $m_8 = \{8, 3, 5\}$ \\
        $m_9 = \{9, 4, 6\}$ \\
        \bottomrule
        \end{tabular}
    \end{minipage}
    \caption{Chords in the (2,5) system.}
\end{table}}

The resulting progression map demonstrates a unique geometric rewiring of the Tonnetz connections.

{\tiny \begin{table}[H]
    \centering
    \begin{minipage}{0.48\textwidth}
        \centering
        \begin{tabular}{l}
        \toprule
        \textbf{Major $\to$ Minor} \\
        \midrule
        $M_0 \to m_0, m_2, m_5$ \\
        $M_1 \to m_1, m_3, m_6$ \\
        $M_2 \to m_2, m_4, m_7$ \\
        $M_3 \to m_3, m_5, m_8$ \\
        $M_4 \to m_4, m_6, m_9$ \\
        $M_5 \to m_5, m_7, m_0$ \\
        $M_6 \to m_6, m_8, m_1$ \\
        $M_7 \to m_7, m_9, m_2$ \\
        $M_8 \to m_8, m_0, m_3$ \\
        $M_9 \to m_9, m_1, m_4$ \\
        \bottomrule
        \end{tabular}
    \end{minipage}\hfill
    \begin{minipage}{0.48\textwidth}
        \centering
        \begin{tabular}{l}
        \toprule
        \textbf{Minor $\to$ Major} \\
        \midrule
        $m_0 \to M_0, M_8, M_5$ \\
        $m_1 \to M_1, M_9, M_6$ \\
        $m_2 \to M_2, M_0, M_7$ \\
        $m_3 \to M_3, M_1, M_8$ \\
        $m_4 \to M_4, M_2, M_9$ \\
        $m_5 \to M_5, M_3, M_0$ \\
        $m_6 \to M_6, M_4, M_1$ \\
        $m_7 \to M_7, M_5, M_2$ \\
        $m_8 \to M_8, M_6, M_3$ \\
        $m_9 \to M_9, M_7, M_4$ \\
        \bottomrule
        \end{tabular}
    \end{minipage}
    \caption{Tonnetz progression map for the (2,5) system.}
\end{table}}

When visualizing this system, we observe a distinct concentration of pathways in the pitch-space layout. The network edges---particularly the $L$-transformations (blue)---cluster tightly around the center, a direct geometric consequence of the narrow $+2$ step size.

\begin{figure}[H]
    \centering
    \begin{minipage}{0.48\textwidth}
        \centering
        \begin{tikzpicture}[scale=0.8, transform shape]
            \tikzset{snode/.style={circle, draw=black, fill=white, inner sep=1.2pt}}
            \def\R{2.0} \def\Rtext{2.3} 
            
            \foreach \s in {0,...,9} {
                \pgfmathsetmacro{\k}{int(\s)}
                \pgfmathsetmacro{\j}{int(mod(\s + 6, 10))}
                \pgfmathsetmacro{\angM}{90 - (\s*2)*18}
                \pgfmathsetmacro{\angm}{90 - (\s*2 + 1)*18}
                \coordinate (M\k) at (\angM:\R);
                \coordinate (m\j) at (\angm:\R);
            }
            \foreach \i in {0,...,9} {
                \draw[red, thick] (M\i) -- (m\i);
                \pgfmathsetmacro{\idxL}{int(mod(\i+6, 10))}
                \draw[blue, thick] (M\i) -- (m\idxL);
                \pgfmathsetmacro{\idxR}{int(mod(\i+5, 10))}
                \draw[levigold, thick] (M\i) -- (m\idxR);
            }
            \foreach \s in {0,...,9} {
                \pgfmathsetmacro{\k}{int(\s)}
                \pgfmathsetmacro{\j}{int(mod(\s + 6, 10))}
                \pgfmathsetmacro{\angM}{90 - (\s*2)*18}
                \pgfmathsetmacro{\angm}{90 - (\s*2 + 1)*18}
          
                \node[snode] at (M\k) {};
                \node[snode] at (m\j) {};
                
                \pgfmathsetmacro{\labelK}{int(mod(7*\k, 10))}
                \pgfmathsetmacro{\labelJ}{int(mod(7*\j, 10))}
                
                \pgfmathsetmacro{\normM}{mod(\angM+360, 360)}
                \pgfmathsetmacro{\normm}{mod(\angm+360, 360)}
                \pgfmathsetmacro{\rotM}{ (\normM > 90.1 && \normM < 269.9) ? \angM - 180 : \angM }
                \pgfmathsetmacro{\rotm}{ (\normm > 90.1 && \normm < 269.9) ? \angm - 180 : \angm }
                
                \pgfmathsetmacro{\leftM}{ (\normM > 90.1 && \normM < 269.9) ? 1 : 0 }
                \ifnum\leftM=1
                    \node[font=\fontsize{8}{9}\selectfont, anchor=east, rotate=\rotM] at (\angM:\Rtext) {$M_{\labelK}$};
                \else
                    \node[font=\fontsize{8}{9}\selectfont, anchor=west, rotate=\rotM] at (\angM:\Rtext) {$M_{\labelK}$};
                \fi

                \pgfmathsetmacro{\leftm}{ (\normm > 90.1 && \normm < 269.9) ? 1 : 0 }
                \ifnum\leftm=1
                    \node[font=\fontsize{8}{9}\selectfont, anchor=east, rotate=\rotm] at (\angm:\Rtext) {$m_{\labelJ}$};
                \else
                    \node[font=\fontsize{8}{9}\selectfont, anchor=west, rotate=\rotm] at (\angm:\Rtext) {$m_{\labelJ}$};
                \fi
            }
        \end{tikzpicture}
        \caption*{Topological Layout \\ (Proof of Isomorphism)}
    \end{minipage}\hfill
    \begin{minipage}{0.48\textwidth}
        \centering
        \begin{tikzpicture}[scale=0.8, transform shape]
            \tikzset{snode/.style={circle, draw=black, fill=white, inner sep=1.2pt}}
            \def\R{2.0} \def\Rtext{2.3} 
            
            \foreach \s in {0,...,9} {
                \pgfmathsetmacro{\k}{int(\s)}
                \pgfmathsetmacro{\j}{int(mod(\s + 6, 10))}
                \pgfmathsetmacro{\angM}{90 - (\s*2)*18}
                \pgfmathsetmacro{\angm}{90 - (\s*2 + 1)*18}
                \coordinate (M\k) at (\angM:\R);
                \coordinate (m\j) at (\angm:\R);
            }
            \foreach \i in {0,...,9} {
                \draw[red, thick] (M\i) -- (m\i);
                \pgfmathsetmacro{\idxL}{int(mod(\i+2, 10))}
                \draw[blue, thick] (M\i) -- (m\idxL);
                \pgfmathsetmacro{\idxR}{int(mod(\i+5, 10))}
                \draw[levigold, thick] (M\i) -- (m\idxR);
            }
            \foreach \s in {0,...,9} {
                \pgfmathsetmacro{\k}{int(\s)}
                \pgfmathsetmacro{\j}{int(mod(\s + 6, 10))}
                \pgfmathsetmacro{\angM}{90 - (\s*2)*18}
                \pgfmathsetmacro{\angm}{90 - (\s*2 + 1)*18}
          
                \node[snode] at (M\k) {};
                \node[snode] at (m\j) {};
                
                \pgfmathsetmacro{\normM}{mod(\angM+360, 360)}
                \pgfmathsetmacro{\normm}{mod(\angm+360, 360)}
                \pgfmathsetmacro{\rotM}{ (\normM > 90.1 && \normM < 269.9) ? \angM - 180 : \angM }
                \pgfmathsetmacro{\rotm}{ (\normm > 90.1 && \normm < 269.9) ? \angm - 180 : \angm }
                
                \pgfmathsetmacro{\leftM}{ (\normM > 90.1 && \normM < 269.9) ? 1 : 0 }
                \ifnum\leftM=1
                    \node[font=\fontsize{8}{9}\selectfont, anchor=east, rotate=\rotM] at (\angM:\Rtext) {$M_{\k}$};
                \else
                    \node[font=\fontsize{8}{9}\selectfont, anchor=west, rotate=\rotM] at (\angM:\Rtext) {$M_{\k}$};
                \fi

                \pgfmathsetmacro{\leftm}{ (\normm > 90.1 && \normm < 269.9) ? 1 : 0 }
                \ifnum\leftm=1
                    \node[font=\fontsize{8}{9}\selectfont, anchor=east, rotate=\rotm] at (\angm:\Rtext) {$m_{\j}$};
                \else
                    \node[font=\fontsize{8}{9}\selectfont, anchor=west, rotate=\rotm] at (\angm:\Rtext) {$m_{\j}$};
                \fi
            }
        \end{tikzpicture}
        \caption*{Pitch-Space Layout \\ (Geometric Rewiring)}
    \end{minipage}
    \caption{Two views of the (2,5) system, mapped via $f(x) \equiv 7x \pmod{10}$. Notice how the Pitch-Space geometry tightly groups the blue $+2$ voice-leading pathways towards the center.}
\end{figure}

\subsubsection{System (4,5): The Final Mirror}

The fourth and final decaphonic universe is generated by the multiplier $f(x) \equiv 9x \pmod{10}$, yielding the $(4,5)$ system. Since $9 \equiv -1 \pmod{10}$, this transformation serves as the ultimate boundary of our 10-TET space. The generator is $q=9$.

{\tiny \begin{table}[H]
    \centering
    \begin{minipage}{0.48\textwidth}
        \centering
        \begin{tabular}{l}
        \toprule
        \textbf{Numerical ($M_r$)} \\
        \midrule
        $M_0 = \{0, 4, 9\}$ \\
        $M_1 = \{1, 5, 0\}$ \\
        $M_2 = \{2, 6, 1\}$ \\
        $M_3 = \{3, 7, 2\}$ \\
        $M_4 = \{4, 8, 3\}$ \\
        $M_5 = \{5, 9, 4\}$ \\
        $M_6 = \{6, 0, 5\}$ \\
        $M_7 = \{7, 1, 6\}$ \\
        $M_8 = \{8, 2, 7\}$ \\
        $M_9 = \{9, 3, 8\}$ \\
        \bottomrule
        \end{tabular}
    \end{minipage}\hfill
    \begin{minipage}{0.48\textwidth}
        \centering
        \begin{tabular}{l}
        \toprule
        \textbf{Numerical ($m_r$)} \\
        \midrule
        $m_0 = \{0, 5, 9\}$ \\
        $m_1 = \{1, 6, 0\}$ \\
        $m_2 = \{2, 7, 1\}$ \\
        $m_3 = \{3, 8, 2\}$ \\
        $m_4 = \{4, 9, 3\}$ \\
        $m_5 = \{5, 0, 4\}$ \\
        $m_6 = \{6, 1, 5\}$ \\
        $m_7 = \{7, 2, 6\}$ \\
        $m_8 = \{8, 3, 7\}$ \\
        $m_9 = \{9, 4, 8\}$ \\
        \bottomrule
        \end{tabular}
    \end{minipage}
    \caption{Chords in the (4,5) system.}
\end{table}}

Completing our decaphonic catalog, we trace the final set of permissible Tonnetz connections. 

{\tiny \begin{table}[H]
    \centering
    \begin{minipage}{0.48\textwidth}
        \centering
        \begin{tabular}{l}
        \toprule
        \textbf{Major $\to$ Minor} \\
        \midrule
        $M_0 \to m_0, m_4, m_5$ \\
        $M_1 \to m_1, m_5, m_6$ \\
        $M_2 \to m_2, m_6, m_7$ \\
        $M_3 \to m_3, m_7, m_8$ \\
        $M_4 \to m_4, m_8, m_9$ \\
        $M_5 \to m_5, m_9, m_0$ \\
        $M_6 \to m_6, m_0, m_1$ \\
        $M_7 \to m_7, m_1, m_2$ \\
        $M_8 \to m_8, m_2, m_3$ \\
        $M_9 \to m_9, m_3, m_4$ \\
        \bottomrule
        \end{tabular}
    \end{minipage}\hfill
    \begin{minipage}{0.48\textwidth}
        \centering
        \begin{tabular}{l}
        \toprule
        \textbf{Minor $\to$ Major} \\
        \midrule
        $m_0 \to M_0, M_6, M_5$ \\
        $m_1 \to M_1, M_7, M_6$ \\
        $m_2 \to M_2, M_8, M_7$ \\
        $m_3 \to M_3, M_9, M_8$ \\
        $m_4 \to M_4, M_0, M_9$ \\
        $m_5 \to M_5, M_1, M_0$ \\
        $m_6 \to M_6, M_2, M_1$ \\
        $m_7 \to M_7, M_3, M_2$ \\
        $m_8 \to M_8, M_4, M_3$ \\
        $m_9 \to M_9, M_5, M_4$ \\
        \bottomrule
        \end{tabular}
    \end{minipage}
    \caption{Tonnetz progression map for the (4,5) system.}
\end{table}}

\begin{figure}[H]
    \centering
    \begin{minipage}{0.48\textwidth}
        \centering
        \begin{tikzpicture}[scale=0.6, transform shape]
            \tikzset{snode/.style={circle, draw=black, fill=white, inner sep=1.2pt}}
            \def\R{2.0} \def\Rtext{2.3} 
            
            \foreach \s in {0,...,9} {
                \pgfmathsetmacro{\k}{int(\s)}
                \pgfmathsetmacro{\j}{int(mod(\s + 6, 10))}
                \pgfmathsetmacro{\angM}{90 - (\s*2)*18}
                \pgfmathsetmacro{\angm}{90 - (\s*2 + 1)*18}
                \coordinate (M\k) at (\angM:\R);
                \coordinate (m\j) at (\angm:\R);
            }
            \foreach \i in {0,...,9} {
                \draw[red, thick] (M\i) -- (m\i);
                \pgfmathsetmacro{\idxL}{int(mod(\i+6, 10))}
                \draw[blue, thick] (M\i) -- (m\idxL);
                \pgfmathsetmacro{\idxR}{int(mod(\i+5, 10))}
                \draw[levigold, thick] (M\i) -- (m\idxR);
            }
            \foreach \s in {0,...,9} {
                \pgfmathsetmacro{\k}{int(\s)}
                \pgfmathsetmacro{\j}{int(mod(\s + 6, 10))}
                \pgfmathsetmacro{\angM}{90 - (\s*2)*18}
                \pgfmathsetmacro{\angm}{90 - (\s*2 + 1)*18}
          
                \node[snode] at (M\k) {};
                \node[snode] at (m\j) {};
                
                \pgfmathsetmacro{\labelK}{int(mod(9*\k, 10))}
                \pgfmathsetmacro{\labelJ}{int(mod(9*\j, 10))}
                
                \pgfmathsetmacro{\normM}{mod(\angM+360, 360)}
                \pgfmathsetmacro{\normm}{mod(\angm+360, 360)}
                \pgfmathsetmacro{\rotM}{ (\normM > 90.1 && \normM < 269.9) ? \angM - 180 : \angM }
                \pgfmathsetmacro{\rotm}{ (\normm > 90.1 && \normm < 269.9) ? \angm - 180 : \angm }
                
                \pgfmathsetmacro{\leftM}{ (\normM > 90.1 && \normM < 269.9) ? 1 : 0 }
                \ifnum\leftM=1
                    \node[font=\fontsize{8}{9}\selectfont, anchor=east, rotate=\rotM] at (\angM:\Rtext) {$M_{\labelK}$};
                \else
                    \node[font=\fontsize{8}{9}\selectfont, anchor=west, rotate=\rotM] at (\angM:\Rtext) {$M_{\labelK}$};
                \fi

                \pgfmathsetmacro{\leftm}{ (\normm > 90.1 && \normm < 269.9) ? 1 : 0 }
                \ifnum\leftm=1
                    \node[font=\fontsize{8}{9}\selectfont, anchor=east, rotate=\rotm] at (\angm:\Rtext) {$m_{\labelJ}$};
                \else
                    \node[font=\fontsize{8}{9}\selectfont, anchor=west, rotate=\rotm] at (\angm:\Rtext) {$m_{\labelJ}$};
                \fi
            }
        \end{tikzpicture}
        \caption*{Topological Layout \\ (Proof of Isomorphism)}
    \end{minipage}\hfill
    \begin{minipage}{0.48\textwidth}
        \centering
        \begin{tikzpicture}[scale=0.6, transform shape]
            \tikzset{snode/.style={circle, draw=black, fill=white, inner sep=1.2pt}}
            \def\R{2.0} \def\Rtext{2.3} 
            
            \foreach \s in {0,...,9} {
                \pgfmathsetmacro{\k}{int(\s)}
                \pgfmathsetmacro{\j}{int(mod(\s + 6, 10))}
                \pgfmathsetmacro{\angM}{90 - (\s*2)*18}
                \pgfmathsetmacro{\angm}{90 - (\s*2 + 1)*18}
                \coordinate (M\k) at (\angM:\R);
                \coordinate (m\j) at (\angm:\R);
            }
            \foreach \i in {0,...,9} {
                \draw[red, thick] (M\i) -- (m\i);
                \pgfmathsetmacro{\idxL}{int(mod(\i+4, 10))}
                \draw[blue, thick] (M\i) -- (m\idxL);
                \pgfmathsetmacro{\idxR}{int(mod(\i+5, 10))}
                \draw[levigold, thick] (M\i) -- (m\idxR);
            }
            \foreach \s in {0,...,9} {
                \pgfmathsetmacro{\k}{int(\s)}
                \pgfmathsetmacro{\j}{int(mod(\s + 6, 10))}
                \pgfmathsetmacro{\angM}{90 - (\s*2)*18}
                \pgfmathsetmacro{\angm}{90 - (\s*2 + 1)*18}
          
                \node[snode] at (M\k) {};
                \node[snode] at (m\j) {};
                
                \pgfmathsetmacro{\normM}{mod(\angM+360, 360)}
                \pgfmathsetmacro{\normm}{mod(\angm+360, 360)}
                \pgfmathsetmacro{\rotM}{ (\normM > 90.1 && \normM < 269.9) ? \angM - 180 : \angM }
                \pgfmathsetmacro{\rotm}{ (\normm > 90.1 && \normm < 269.9) ? \angm - 180 : \angm }
                
                \pgfmathsetmacro{\leftM}{ (\normM > 90.1 && \normM < 269.9) ? 1 : 0 }
                \ifnum\leftM=1
                    \node[font=\fontsize{8}{9}\selectfont, anchor=east, rotate=\rotM] at (\angM:\Rtext) {$M_{\k}$};
                \else
                    \node[font=\fontsize{8}{9}\selectfont, anchor=west, rotate=\rotM] at (\angM:\Rtext) {$M_{\k}$};
                \fi

                \pgfmathsetmacro{\leftm}{ (\normm > 90.1 && \normm < 269.9) ? 1 : 0 }
                \ifnum\leftm=1
                    \node[font=\fontsize{8}{9}\selectfont, anchor=east, rotate=\rotm] at (\angm:\Rtext) {$m_{\j}$};
                \else
                    \node[font=\fontsize{8}{9}\selectfont, anchor=west, rotate=\rotm] at (\angm:\Rtext) {$m_{\j}$};
                \fi
            }
        \end{tikzpicture}
        \caption*{Pitch-Space Layout \\ (Geometric Rewiring)}
    \end{minipage}
    \caption{Two views of the (4,5) system. Driven by $f(x) \equiv 9x \pmod{10}$, this universe demonstrates the final orientation-preserving permutation of the decaphonic geometry.}
\end{figure}

\begin{conclusion}[The Musician's Playbook: Decaphonic Improvisation]
The 10-TET universe offers a profound pedagogical advantage over the traditional 12-TET landscape. As proven mathematically in Theorem 3, there is absolutely no ``Negative Harmony'' in this space. Every single isomorphic transformation is strictly orientation-preserving, meaning a Major chord inevitably remains a Major chord, and the inherent syntax of voice-leading paths never forces you to invert your physical reflexes.

For the performing musician, this simplifies decaphonic navigation immensely. You must only train your muscle memory to master the physical hand shapes (the ``grips") and the straightforward geometric steps of the foundational $(6,5)$ system. Once those are internalized, exploring the alien terrains of $(8,5)$, $(2,5)$, or $(4,5)$ requires no new conceptual logic. To transpose a progression into one of these expanded universes, you simply multiply your root notes by $3$, $7$, or $9$ respectively. The abstract musical journey remains exactly the same---only your hands must perform wider physical leaps across the keyboard to execute the rewired geometry.
\end{conclusion}

\subsubsection{Acoustic Realities: Auditory Proofs of the Decaphonic Isomorphisms}

To verify that the microtonal landscapes of 10-TET constitute genuinely viable compositional environments rather than mere algebraic abstractions, we designed a short, 16-chord cyclical progression titled \textit{Decaphonic Miniature}. Because the fundamental interval unit in a decaphonic tuning system equals exactly $120$ cents, executing chords in close, block-like positions can cause severe acoustic masking and harsh sensory roughness due to the narrowness of the steps. To resolve this and fully exploit the physical modeling synthesis of the piano strings, the composition is arranged using a wide, spatial open-voicing layout. The constituent triads are dispersed across three distinct octaves---anchored by a deep fundamental bass---while explicit sustain pedal messages are engaged to trigger the instrument's sympathetic resonance, allowing the microtonal frequencies to form a crystalline sonic cloud.

Rather than standing still, the progression executes an exhaustive periodic cycle that wends across the bipartite Levi graph of the foundational $(6,5)$ anchor space. The 16-chord trajectory is explicitly mapped as follows:
{\tiny \begin{equation}
 M_0 \to m_0 \to M_4 \to m_9 \to M_9 \to m_5 \to M_0 \to m_6 \to M_1 \to m_7 \to M_2 \to m_2 \to M_6 \to m_1 \to M_5 \to m_0 \to M_0
\end{equation}}

This closed cycle tests all allowed transformation vectors within the non-degenerate pitch space, moving through structural parallel, left, and right voice-leading connections before seamlessly returning to the origin. To observe how this continuous progression maps onto our geometric coordinate systems, the path is tracked and highlighted explicitly across the four isomorphic decaphonic unrollings in Figure \ref{fig:sonata_10tet_paths}.

\newcommand{\drawSonataGraph}[1]{
    \begin{tikzpicture}[scale=0.6, transform shape]
        \tikzset{
            basenode/.style={circle, draw=black, fill=white, inner sep=1.2pt},
            startnode/.style={circle, draw=black, fill=darkgreen, inner sep=1.5pt, line width=0.7pt},
            midarrow/.style={decoration={markings, mark=at position 0.75 with {\arrow[orange]{Stealth[length=2.5mm, width=0.8mm]}}}, postaction={decorate}}
        }
        \def\R{2.0} \def\Rtext{2.3} 
        
        \foreach \s in {0,...,9} {
            \pgfmathsetmacro{\k}{int(\s)}
            \pgfmathsetmacro{\j}{int(mod(\s + 6, 10))}
            \pgfmathsetmacro{\angM}{90 - (\s*2)*18}
            \pgfmathsetmacro{\angm}{90 - (\s*2 + 1)*18}
            \coordinate (M\k) at (\angM:\R);
            \coordinate (m\j) at (\angm:\R);
        }
        
        \foreach \i in {0,...,9} {
            \draw[red, thin, opacity=0.15] (M\i) -- (m\i);
            \pgfmathsetmacro{\idxL}{int(mod(\i+6, 10))}
            \draw[blue, thin, opacity=0.15] (M\i) -- (m\idxL);
            \pgfmathsetmacro{\idxR}{int(mod(\i+5, 10))}
            \draw[levigold, thin, opacity=0.15] (M\i) -- (m\idxR);
        }
        
        \foreach \u/\v in {M0/m0, m0/M4, M4/m9, m9/M9, M9/m5, m5/M0, M0/m6, m6/M1, M1/m7, m7/M2, M2/m2, m2/M6, M6/m1, m1/M5, M5/m0, m0/M0} {
            \path[draw=orange, line width=1.2pt, midarrow] (\u) -- (\v);
        }
        
        \foreach \s in {0,...,9} {
            \pgfmathsetmacro{\k}{int(\s)}
            \pgfmathsetmacro{\j}{int(mod(\s + 6, 10))}
            \node[basenode] at (M\k) {};
            \node[basenode] at (m\j) {};
        }
        \node[startnode] at (M0) {};
        
        \foreach \s in {0,...,9} {
            \pgfmathsetmacro{\k}{int(\s)}
            \pgfmathsetmacro{\j}{int(mod(\s + 6, 10))}
            \pgfmathsetmacro{\angM}{90 - (\s*2)*18}
            \pgfmathsetmacro{\angm}{90 - (\s*2 + 1)*18}
            
            \pgfmathsetmacro{\labelK}{int(mod(#1*\k, 10))}
            \pgfmathsetmacro{\labelJ}{int(mod(#1*\j, 10))}
            
            \pgfmathsetmacro{\normM}{mod(\angM+360, 360)}
            \pgfmathsetmacro{\normm}{mod(\angm+360, 360)}
            \pgfmathsetmacro{\rotM}{ (\normM > 90.1 && \normM < 269.9) ? \angM - 180 : \angM }
            \pgfmathsetmacro{\rotm}{ (\normm > 90.1 && \normm < 269.9) ? \angm - 180 : \angm }
            
            \pgfmathsetmacro{\leftM}{ (\normM > 90.1 && \normM < 269.9) ? 1 : 0 }
            \ifnum\leftM=1
                \node[font=\fontsize{7}{8}\selectfont, anchor=east, rotate=\rotM] at (\angM:\Rtext) {$M_{\labelK}$};
            \else
                \node[font=\fontsize{7}{8}\selectfont, anchor=west, rotate=\rotM] at (\angM:\Rtext) {$M_{\labelK}$};
            \fi

            \pgfmathsetmacro{\leftm}{ (\normm > 90.1 && \normm < 269.9) ? 1 : 0 }
            \ifnum\leftm=1
                \node[font=\fontsize{7}{8}\selectfont, anchor=east, rotate=\rotm] at (\angm:\Rtext) {$m_{\labelJ}$};
            \else
                \node[font=\fontsize{7}{8}\selectfont, anchor=west, rotate=\rotm] at (\angm:\Rtext) {$m_{\labelJ}$};
            \fi
        }
    \end{tikzpicture}
}

\begin{figure}[H]
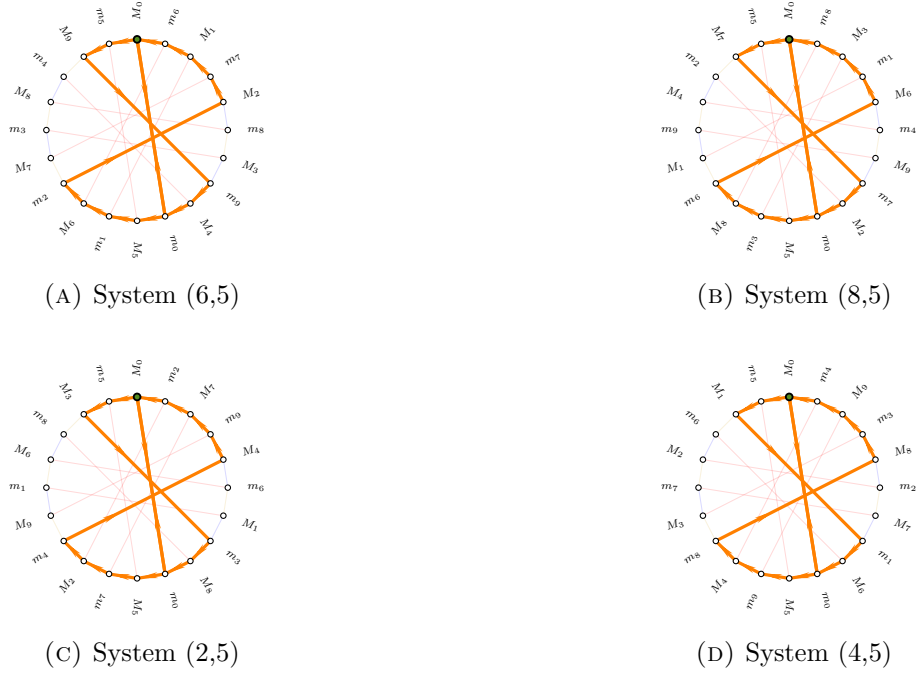

    \centering
    \begin{subfigure}{0.48\textwidth}
        \centering
        \drawSonataGraph{1}
        \caption{System (6,5)}
    \end{subfigure}
    \hfill
    \begin{subfigure}{0.48\textwidth}
        \centering
        \drawSonataGraph{3}
        \caption{System (8,5)}
    \end{subfigure}
    
    \vspace{0.5cm}
    
    \begin{subfigure}{0.48\textwidth}
        \centering
        \drawSonataGraph{7}
        \caption{System (2,5)}
    \end{subfigure}
    \hfill
    \begin{subfigure}{0.48\textwidth}
        \centering
        \drawSonataGraph{9}
        \caption{System (4,5)}
    \end{subfigure}
    \caption{\small The 16-chord Decaphonic Miniature mapped identically across the four non-degenerate 10-TET Tonnetz networks in their purely isomorphic topology. The orange line follows the invariant structural path of the composition. Elegant arrows dictate the flow of the progression. The anchor triad ($M_0$ in foundational coordinates) is highlighted in green, initiating and concluding the closed cycle. Although the physical pitch classes predictably permute by $a \in \{1, 3, 7, 9\}$, this visualization rigorously proves the absolute structural identity of the composition across all parallel acoustic dimensions.}
    \label{fig:sonata_10tet_paths}
\end{figure}

The auditory results of this multi-dimensional composition offer a striking acoustic verification of our architectural framework. While the physical interval shapes drastically dilate and contract across the expansions due to the scalar multipliers, the foundational path logic and voice-leading syntax remain identically preserved. 

To provide a comprehensive listening experience, we have compiled these parallel acoustic dimensions into a single, continuous audio demonstration. The recording executes the exact 16-chord trajectory sequentially across the four isomorphic universes, beginning with the foundational baseline (6,5), followed by the first expansion (8,5), the second expansion (2,5), and concluding with the final mirror (4,5). Brief announcements separate each topological dimension. We invite the reader to explore this auditory proof via the link below:

\vspace{0.3cm}
\begin{itemize}
    \item \textbf{Decaphonic Miniature: Sequential Demonstration}\\
    \href{https://www.fuw.edu.pl/~nurowski/Decaphonic_Jingle_all_systems.wav}{Listen to all four systems (.wav)}
\end{itemize}

\end{document}